\documentclass[12pt]{article}
\usepackage{amsmath,amsthm,amssymb}
\usepackage[hang]{caption2}
\usepackage{indentfirst}
\usepackage[section]{placeins}
\numberwithin{equation}{section}     
\usepackage{graphics}
\usepackage{pdfsync}
\usepackage{hyperref}

\def\accentsfrancais{applemac}
  \RequirePackage[\accentsfrancais]{inputenc}

\newtheorem{thm}{Theorem}
\newtheorem{theorem}[thm]{Theorem}
\numberwithin{thm}{section}
\newtheorem{lemma}[thm]{Lemma}
\newtheorem{cor}[thm]{Corollary}

\newtheorem{remark}[thm]{Remark}

\def\ds{\displaystyle}
\def\emptyset{/\kern-.51em o}
\def\eq{\mathop{\vrule height2,6pt depth-2,3pt
         width -1pt\kern 0pt =}}

\let\norbali\normalbaselines
\def\anorbali{\norbali\advance\lineskip\jot
\advance\baselineskip\jot\advance\lineskiplimit\jot}
\def\ouvre{\let\normalbaselines\anorbali}
\def\D{{\mathbb{D}}}
\def\B{{\mathbb{B}}}

\def\D{{\mathbb{D}}}
\def\R{{\mathbb{R}}}
\def\U{{\mathbb{U}}}
\def\V{{\mathbb{V}}}
\def\W{{\mathbb{W}}}
\def\P{{\mathbb{P}}}

\def\Z{{\mathbb{Z}}}

\def\Ge{{\bf e}}

\def\Gi{{\bf i}}

\def\GH{{\bf H}}
\def\GI{{\bf I}}

\def\GR{{\bf R}}

\def\GT{{\bf T}}

\def\eps{\varepsilon}

\begin{document}
\title{Error estimates in periodic homogenization with a non-homogeneous Dirichlet condition. }

\author{G. Griso}
\date{}
\maketitle
 
{\footnotesize
\begin{center}
 Laboratoire J.-L. Lions--CNRS, Bo\^\i te courrier 187, Universit\'e  Pierre et
Marie Curie,\\ 4~place Jussieu, 75005 Paris, France, \; Email: griso@ann.jussieu.fr\\
\end{center} }

\def\D{{\mathbb{D}}}
\def\B{{\mathbb{B}}}

\def\D{{\mathbb{D}}}
\def\R{{\mathbb{R}}}
\def\U{{\mathbb{U}}}
\def\V{{\mathbb{V}}}
\def\W{{\mathbb{W}}}
\def\P{{\mathbb{P}}}

\def\liminf{\mathop{\underline{\rm lim}}}
\def\limsup{\mathop{\overline{\hbox{\rm lim}}}}
\def\sym{\fam\comfam\com}
\font\tensym=msbm10
\font\sevensym=msbm7
\font\fivesym=msbm5
%

\let\ds\displaystyle
\def\eps{\varepsilon}

\begin{abstract}
In this paper we investigate the homogenization problem with a non-homogeneous Dirichlet condition. Our aim is to give error estimates with  boundary data in $H^{1/2}(\partial \Omega)$. The tools used are those of the unfolding method in periodic homogenization.
\end{abstract}

\section{Introduction}

\noindent We consider  the following  homogenization problem: 
\begin{equation}\label{1.0}
\phi^\eps\in H^1(\Omega),\qquad  -\hbox{div}\big( A_\eps\nabla\phi^\eps)=f\qquad \hbox{in}\quad\Omega,\qquad \phi^\eps=g\qquad \hbox{on}\quad \partial\Omega
\end{equation} where  $A_\eps$ is a periodic matrix satisfying  the usual condition of uniform ellipticity and where $f\in L^2(\Omega)$ and $g\in H^{1/2}(\partial\Omega)$\footnote{The homogenization problem with  $L^p$ boundary data is investigated in \cite{AVE}.}. We know (see e.g. \cite{Ben}, \cite{CDG}, \cite{DoDo}) that the function $\phi^\eps$  weakly converges  in $H^1(\Omega)$ towards the solution $\Phi$ of the homogenized problem
\begin{equation}\label{0000}
\Phi\in H^1(\Omega),\qquad  -\hbox{div}\big( {\cal A}\nabla\Phi)=f\qquad \hbox{in}\quad\Omega,\qquad \Phi=g\qquad \hbox{on}\quad \partial\Omega
\end{equation} where ${\cal A}$ is the homogenized matrix (see \eqref{Homo} and \eqref{Aij}). Using the results in \cite{CDG} we can give an approximation of $\phi^\eps$ belonging to $H^1(\Omega)$ and we easily obtain
$$\phi^\eps-\Phi-\eps\sum_{i=1}^n {\cal Q}_\eps\Bigl({\partial\Phi\over\partial x_i}\Bigr)\chi_i\Big({.\over\eps}\Big)\longrightarrow 0\quad \hbox{strongly in }\quad H^1(\Omega)$$ where  ${\cal Q}_\eps$ is the {\it scale-splitting operator} (see \cite{CDG} or Subsection \ref{RUM}) and where the $\chi_i$ are the correctors (see \eqref{cor}). 

One of the aim of this paper is to give error estimates for this homogenization problem. Obviously, if we have $g\in H^{3/2}(\partial \Omega)$ and the appropriate assumptions on the boundary of the domain then we can  apply the results in \cite{Ben}, \cite{DoDo}, \cite{GG0}, \cite{GG1}, \cite{GG2} and \cite{MOVO} to deduce error estimates.  All of them require that the function $\Phi$ belongs at least to $H^2(\Omega)$. Here, the solution $\Phi$ of the homogenized problem \eqref{0000} is only in $H^1(\Omega)\cap H^2_{loc}(\Omega)$. In this paper we have to deal with this lack of regularity; this is the main difficulty. 

The tools of the unfolding method in periodic homogenization to obtain error estimates (see \cite{GG0}, \cite{GG1} and \cite{GG2}) are the  projection theorems. This is why we prove two new projection theorems; the Theorems \ref{Theorem 2.2} and \ref{Theorem 2.3 : }.   Here, both theorems concern the functions  $\phi\in H^1_0(\Omega)$ satisfying $ \nabla \phi/\rho\in L^2(\Omega;\R^n)$ where $\rho(x)$ is the distance between $x$ and the boundary of $\Omega$. In the first one we give the distance between ${\cal T}_\eps(\phi)$ (see \cite{CDG} or Subsection \ref{Some reminds} for the definition of the unfolding operator ${\cal T}_\eps$) and the space $L^2(\Omega; H^1_{per}(Y))$ in terms of the $L^2$ norms of $\phi/\rho$ and $ \nabla \phi/\rho$ and obviously $\eps$.  In the second one we prove an upper bound for the distance between ${\cal T}_\eps(\nabla\phi)$ and the space $\nabla H^1(\Omega)\oplus \nabla_y L^2(\Omega; H^1_{per}(Y))$;  again  in terms of the $L^2$ norms of $\phi/\rho$ and $ \nabla \phi/\rho$ and  $\eps$  (see Section \ref{TNPT}). This last theorem is partially a consequence of the first one. In this paper we derive the new error estimates from the second projection theorem and those obtained in \cite{GG2}. 

Different results are known about the global $H^1$ error estimate regarding  the classical homogenization problem \eqref{1.0} (see e.g. \cite{Ben}, \cite{DoDo}). Those with the minimal assumptions are given in \cite{GG1}; if the solution of  the homogenized problem \eqref{0000} belongs to $H^2(\Omega)$ -see Proposition 4.3 in \cite{GG1}- (respectively $H^{3/2}(\Omega)$; see Theorem 3.3 in \cite{GG2})  then the $H^1$ global error is of order $\eps^{1/2}$ (resp.  $\eps^{1/4}$) while if this solution belongs to $H^2_{loc}(\Omega)\cap W^{1,p}(\Omega)$ ($p>2$) the obtained $H^1$ global error is smaller and depends on $p$ (see Proposition 4.4 in \cite{GG1})\footnote{These propositions or theorem are proved with a Dirichlet condition, with a non-homogenous Dirichlet condition belonging to $H^{3/2}(\partial\Omega)$ the results are obviously the same.}. Here, with a non-homogeneous Dirichlet condition belonging only to $H^{1/2}(\partial \Omega)$ we do not obtain a global $H^1$ error estimate. The $L^2$ global error estimate only requires a boundary of $\Omega$ sufficiently smooth (of class ${\cal C}^{1,1}$) or a convex open set. Obviously if it is possible to make use of a global $H^1$ error estimate, the $L^2$ global error will be better (the reader will be able to compare the Theorem 3.2 in \cite{GG2} with  the Theorem \ref{TH3}). The $H^1$ local error estimate is always linked to the $L^2$ global error and never needs more assumption (see Theorem 3.2 in \cite{GG2} or the proof of Theorem \ref{TH1}).


The paper is organized as follows. In Section 2 we introduce a few general notations, then we give  some reminds\footnote{We want to simplify the reading to a non-familiar reader with the unfolding method} on  lemmas, definitions and results about the unfolding method in periodic homogenization (see \cite{CDG}), then we  prove some new results involving the main operators of this method. Section 3 is devoted to the new projection theorems. In Section 4, we recall the main results on the classical homogenization problem.  In Section 5  we introduce an operator  which allows to lift  the distributions belonging to $H^{-1/2}(\partial \Omega)$ in functions belonging to $L^2(\Omega)$; this lifting operator will play an important role in the case of strongly oscillating boundary data.  In Section 6 we derive the  error estimates results (Theorems \ref{TH1} and  \ref{TH3}) with a non-homogenous Dirichlet condition. We end the paper by investigating a case where the boundary data are  strongly oscillating (see Theorem \ref{TH65} in Section 7). A forthcoming paper will be devoted to homogenization problems with other strongly oscillating boundary data.

As general references on the homogenization theory we refer to \cite{AL1}, \cite{Ben} and \cite{DoDo}. The reader is referred to \cite{CDG},  \cite{CDDGZ} and \cite{DoDo} for an introduction of the unfolding method in periodic homogenization.  The following papers  \cite{BGG}, \cite{BG3}, \cite{BLG}, \cite{BGM}, \cite {CDGO}, \cite{GR}, \cite{OZBGMR} give various applications  of the unfolding method in periodic homogenization.  As far as the error estimates  are concerned, we refer to \cite{ALAM}, \cite{Ben}, \cite{GG0}, \cite{GG1}, \cite{GG2},  \cite{KES}, \cite{MOVO} and  \cite{DOBV}. 
\vskip 1mm
\noindent{\bf Keywords: } periodic homogenization, error estimate, non-homogeneous Dirichlet condition, periodic unfolding method.

\noindent{\bf  Mathematics Subject Classification (2000)}: 35B27, 65M15, 74Q15. 

%

\section{Preliminaries}

\subsection{Notations}\label{not}

\noindent $\bullet$  The space $\R^k$ ($k\ge 1$) is endowed with the standard basis $\big(\Ge_1,\ldots,\Ge_k\big)$; the euclidian norm is denoted $|\cdot|$.

\noindent $\bullet$ We denote by $\Omega$  a bounded domain in $\R^n$ with a Lipschitz boundary.\footnote{In Section \ref{Op} and those which follow, we will assume that $\Omega$ is a bounded domain of class ${\cal C}^{1,1}$ or an open bounded convex set.} Let $\rho(x)$ be the distance between $x\in \R^n$ and the boundary of $\Omega$, we set
\begin{equation*}
 \widetilde{\Omega}_\gamma=\Big\{x\in \Omega\;|\; \rho(x)<\gamma\Big\}\qquad \widetilde{\widetilde{\Omega}}_\gamma=\Big\{x\in \R^n\;|\; \rho(x)<\gamma\Big\}\qquad \gamma\in \R^{*+}. 
\end{equation*}
\noindent $\bullet$ There exist constants $a$, $A$ and $\gamma_0$ strictly positive and $M\ge 1$, a finite number $N$ of local euclidian coordinate systems  $(O_r;\Ge_{1r},\ldots,\Ge_{nr})$ and mappings $f_r\,: \, [-a,a]^{n-1}\longrightarrow \R$, Lipschitz continuous with ratio $M$, $1\le r\le N$, such that  (see e.g. \cite{GG3} or \cite{GDecomp})
\begin{equation}\label{bound}
\begin{aligned}
& \partial \Omega=\displaystyle\bigcup_{r=1}^N\Bigl\{x=x^{'}_r+x_{nr}\Ge_{nr}\in\R^n\;\;|\; \; x^{'}_r\in\Delta_a\;\;\hbox{and}\;\;  x_{nr}=f_r(x^{'}_{r}) \Bigr\},\\
& \hbox{where} \enskip x^{'}_r=x_{1r}\Ge_{1r}+\ldots + x_{n-1 r}\Ge_{n-1 r},\quad \Delta_a=\Big\{x^{'}_r \;|\; x_{ir}\in ]-a,a[,\; i\in \{1,\ldots,n-1\}\Big\}\\
&\widetilde{\Omega}_{\gamma_0}\subset \bigcup_{r=1}^N\Omega_r\subset\Omega,\qquad \Omega_r=\Bigl\{x \in\R^n\;|\; x^{'}_r\in \Delta_a\;\hbox{and}\;
f_r(x^{'}_r)<x_{nr}<f_r(x^{'}_{r})+A\Bigr\}\\
&\widetilde{\widetilde{\Omega}}_{\gamma_0}\subset \bigcup_{r=1}^N\Bigl\{x \in\R^n\;|\; x^{'}_r\in \Delta_a\;\;\hbox{and}\;\;
f_r(x^{'}_r)-A<x_{nr}<f_r(x^{'}_{r})+A\Bigr\}\\
&\forall r\in\{1,\ldots,N\},\quad \forall x\in \Omega_r\quad \hbox{we have}\quad  {1\over 2M}(x_{nr}-f_r(x^{'}_r))\le \rho(x)\le x_{nr}-f_r(x^{'}_r).
\end{aligned}\end{equation}
$\bullet$ We set
\begin{equation*}
\begin{aligned}
& Y=]0,1[^n,\qquad \Xi_\eps=\bigl\{\xi\in \Z^n\; |\; \eps(\xi+ Y)\subset  \Omega \bigr\},\\
&\widehat{\Omega}_\eps =\hbox{interior}\Bigl(\bigcup_{\xi\in\Xi_\eps}\eps(\xi+\overline{Y})\Bigr),\qquad\Lambda_\eps=\Omega\setminus \widehat{\Omega}_\eps,
\end{aligned}\end{equation*} where   $\eps$ is a strictly positive real. 

\noindent $\bullet$  We define 

$\star$ $H^1_{\rho}(\Omega)=\Big\{\phi\in  L^2(\Omega)\;|\; \rho\nabla\phi\in L^2(\Omega;\R^n)\Big\}$,

%

$\star$ $\ds L^2_{1/\rho}(\Omega)=\Big\{\phi\in L^2(\Omega)\;|\; {\phi/\rho} \in L^2(\Omega)\Big\}$,

$\star$ $\ds H^1_{1/\rho}(\Omega)=\Big\{\phi\in H^1_0(\Omega)\;|\; {\nabla\phi/\rho}\in L^2(\Omega;\R^n)\Big\}$.

\noindent We endow $H^1_{\rho}(\Omega)$ (resp.   $H^1_{1/\rho}(\Omega)$) with the norm
\begin{equation*}
\begin{aligned}
&\forall \phi\in H^1_{\rho}(\Omega),\quad||\phi||_\rho=||\phi||_{L^2(\Omega)}+||\rho\nabla\phi||_{L^2(\Omega;\R^n)}\\
\hbox{( resp. }\;\;&\forall \phi\in  H^1_{1/\rho}(\Omega),\quad ||\phi||_{1/\rho}=\bigl\|{\nabla\phi /\rho}\big\|_{ L^2(\Omega;\R^n)}\;\hbox{).}
\end{aligned}\end{equation*} 
Note that if $\phi$ belongs to $H^1_{\rho}(\Omega)$ then the function $\psi=\rho\phi$ is in $H^1_0(\Omega)$ and vice versa if a function $\psi$ belongs to $H^1_0(\Omega)$ then $\phi=\ds{\psi / \rho}$ is in $H^1_\rho(\Omega)$ since we have (see \cite{BR} or \cite{LM})
\begin{equation}\label{200}
\forall\psi\in H^1_0(\Omega),\qquad \big\|{\psi /\rho}\big\|_{L^2(\Omega)}\le C||\nabla \psi||_{L^2(\Omega;\R^n)}.
\end{equation}
\vskip 1mm
Below we recall  a classical extension lemma which is proved for example in \cite{GG1} or which can be proved using the local charts \eqref{bound}.

\begin{lemma}\label{lemP} Let $\Omega$ be a bounded domain with a Lipschitz boundary, there exist $c_0\ge 1$ (which depends only on the boundary of $\Omega$) and  a linear and continuous extension operator ${\cal P}$ from $L^2(\Omega)$ into $L^2(\R^n)$ which also maps $H^1(\Omega)$ into $H^1(\R^n)$ such that
\begin{equation}\label{lemPEst}
\begin{aligned}
\forall\phi\in L^2(\Omega),\quad {\cal P}(\phi)_{|_\Omega}=\phi,\qquad & || {\cal P}(\phi)||_{L^2(\R^n)}\le C ||\phi||_{L^2(\Omega)},\\
&|| {\cal P}(\phi)||_{L^2(\widetilde{\widetilde{\Omega}}_{\gamma})}\le C ||\phi||_{L^2(\widetilde{\Omega}_{c_0\gamma})}.
\end{aligned}
\end{equation} Moreover we have
\begin{equation*}
\forall\phi\in H^1(\Omega),\qquad  ||\nabla{\cal P}(\phi)||_{L^2(\R^n ; \R^n)}\le C ||\nabla\phi||_{L^2(\Omega ; \R^n)}.
\end{equation*} 
\end{lemma} 
 {\it From now on, if need be, a function $\phi$ belonging to  $L^2(\Omega)$ (resp. $H^1(\Omega)$) will be extended to a function belonging to  $L^2(\R^n)$ (resp. $H^1(\R^n)$) using the above lemma. The extension will be still denoted $\phi$.}
\vskip 2mm

\subsection{A characterization of the functions belonging to $H^1_{1/\rho}(\Omega)$}\label{sub22}

The two first projection theorems (see \cite{GG1}) regarded the functions belonging to $H^1_0(\Omega)$  while those  in  \cite{GG2} regarded the functions in $H^1(\Omega)$. In this paper we prove two new projection theorems which involve the functions in $\ds H^1_{1/\rho}(\Omega)$; this is why we first give a simple characterization of these functions in  the  Lemma \ref{lem21} below.
\vskip 1mm
Observe first that if a function $\phi$ satisfies $\phi/\rho\in H^1_0(\Omega)$ then $\phi$ belongs to $H^1_{1/\rho}(\Omega)$. The reverse is true.%

\begin{lemma}\label{lem21} Let $\Omega$ be  a bounded domain with a Lipschitz boundary, we have
$$\phi\in H^1_{1/\rho}(\Omega)\;\; \Longleftrightarrow \quad {\phi / \rho}\in H^1_0(\Omega).$$ Furthermore  there exists a constant which depends only on $\partial \Omega$ such that
\begin{equation}\label{f0}
\forall\phi\in H^1_{1/\rho}(\Omega)\qquad \bigl\|{\phi/ \rho^2}\big\|_{L^2(\Omega)}+\bigl\|{\phi/\rho}\big\|_{H^1(\Omega)}\le C||\phi||_{1/\rho}.
\end{equation} 
\end{lemma}
\begin{proof}  {\it Step 1. }  Let $\phi$ be in $H^1(]-a,a[^{n-1}\times ]0,A[)$ ($a, \; A>0$) satisfying $\ds {1\over x_n}\nabla \phi(x)\in L^2(]-a,a[^{n-1}\times ]0,A[)$ and $\phi(x)=0$ for a.e. $x$ in $]-a,a[^{n-1}\times\{0\}\cup]-a,a[^{n-1}\times \{A\}$. 

\noindent We have
\begin{equation}\label{f1}
\int_{]-a,a]^{n-1}\times ]0,A[} {|\phi(x)|^2\over x_n^4}dx\le{1\over 2}\int_{]-a,a]^{n-1}\times ]0,A[}{|\nabla\phi(x)|^2\over x_n^2}dx.
\end{equation}
To prove \eqref{f1}, we choose $\eta>0$ and we integrate by parts
$\ds \int_{]-a,a]^{n-1}\times ]0,A[} {|\phi(x)|^2\over (\eta+x_n)^4}dx$, then thanks to the identity relation $2bc\le b^2+c^2$ we obtain
$$\begin{aligned}
\int_{]-a,a]^{n-1}\times ]0,A[} {|\phi(x)|^2\over (\eta+x_n)^4}dx & \le{1\over 2}\int_{]-a,a]^{n-1}\times ]0,A[}{1\over (\eta+x_n)^2}\Big|{\partial \phi\over \partial x_n}(x)\Big|^2dx\\
& \le{1\over 2}\int_{]-a,a]^{n-1}\times ]0,A[}{|\nabla\phi(x)|^2\over x_n^2}dx.
\end{aligned}$$ Passing to the limit $(\eta\to 0$) it leads to \eqref{f1}.
\vskip 1mm
\noindent{\it Step 2. } Let $h$ be in $W^{1,\infty}(\Omega)$ such that
$$\forall x\in \Omega, \qquad \begin{aligned}
&h(x)\in [0,1],\\
&h(x)=1\quad \hbox{if}\quad \rho(x)\ge \gamma_0,\\
&h(x)=0\quad \hbox{if}\quad \rho(x)\le {\gamma_0/2}.\\
\end{aligned}$$ Let $\phi$ be  in $H^1_{1/\rho}(\Omega)$. The function $\phi h/\rho^4$ belongs to  $H^1_0(\Omega)$, therefore  as a consequence of the Poincar's inequality we obtain
\begin{equation}\label{f2}
\begin{aligned}
\int_{\Omega} {|\phi(x)h(x)|^2\over \rho(x)^4}dx&\le C\int_{\Omega}\Big|\nabla\Big({\phi(x) h(x)\over \rho(x)^4}\Big)\Big|^2dx\le  C\int_{\Omega}\big(|\nabla\phi(x)|^2+|\phi(x)|^2\big)dx\\
&\le C\int_{\Omega}|\nabla\phi(x)|^2dx\le C\int_{\Omega}{|\nabla\phi(x)|^2\over \rho(x)^2}dx.
\end{aligned}\end{equation} 
Then using the local chart  of $\Omega_r$  given by \eqref{bound}, the inequality \eqref{f1} and thanks to a simple change of variables we get
$$\int_{\Omega_r} {|\phi(x)\big(1-h(x)\big)|^2\over \rho(x)^4}dx \le C\int_{\Omega_r}{|\nabla\big(\phi(x)(1-h(x)\big)|^2\over \rho(x)^2}dx\le C\int_{\Omega_r}{|\nabla\phi(x)|^2+|\phi(x)|^2\over \rho(x)^2}dx.$$ Since $\phi\in H^1_0(\Omega)$ the function  $\phi/\rho$ belongs to $L^2(\Omega)$ and  we have \eqref{200}. Hence, adding these inequalities ($r=1,\ldots,N$) we obtain
\begin{equation}\label{f3}
\int_{\Omega} {|\phi(x)\big(1-h(x)\big)|^2\over \rho(x)^4}dx \le  C\int_{\Omega}{|\nabla\phi(x)|^2\over \rho(x)^2}dx.
\end{equation} Finally $\phi/\rho^2\in L^2(\Omega)$ and \eqref{f2}-\eqref{f3} lead to  $\ds \bigl\|{\phi/\rho^2}\big\|_{L^2(\Omega)}\le C||\phi||_{1/\rho}$ and then \eqref{f0}.
\end{proof}

\subsection{Two   lemmas} 

In the Lemma \ref{lemA} we give sharp estimates of a function on the boundary and in a neighborhood of the boundary of $\Omega$. The second estimate in \eqref{403} is used to obtain the $L^2$ global error.
\vskip 1mm
\begin{lemma}\label{lemA}  Let $\Omega$ be  a bounded domain with a Lipschitz boundary, there  exists $\gamma_0>0$ (see Subsection \ref{sub22}) such that for  any $\gamma\in ]0,\gamma_0]$ and  for any $\phi\in H^1(\Omega)$ we have
\begin{equation}\label{403}
\begin{aligned}
&||\phi||_{L^2(\partial \Omega)}\le {C\over \gamma^{1/2}}\big(||\phi||_{L^2(\widetilde{\Omega}_\gamma)}+\gamma||\nabla \phi||_{L^2(\widetilde{\Omega}_\gamma ; \R^n)}\big),\\
&||\phi||_{L^2(\widetilde{\Omega}_\gamma)}\le C \big(\gamma^{1/2}||\phi||_{L^2(\partial\Omega)}+\gamma||\nabla \phi||_{L^2(\widetilde{\Omega}_\gamma ; \R^n)}\big).
\end{aligned}\end{equation} 
The constants do not depend  on $\gamma$. 
\end{lemma}
\begin{proof} Let  $\psi$ be in $H^1(]-a,a[^{n-1}\times ]0,A[)$. For $\eta\in ]0,A[$ we have
\begin{equation*}
\begin{aligned}
||\psi||^2_{L^2(]-a,a[^{n-1}\times\{0\})} & \le {C\over \eta}||\psi||^2_{L^2(]-a,a[^{n-1}\times ]0,\eta[)}+C\eta||\nabla \psi||^2_{L^2(]-a,a[^{n-1}\times ]0,\eta[;\R^n)},\\
||\psi||^2_{L^2(]-a,a[^{n-1}\times]0,\eta[)} & \le {C \eta}||\psi||^2_{L^2(]-a,a[^{n-1}\times \{0\})}+C\eta^2||\nabla \psi||^2_{L^2(]-a,a[^{n-1}\times ]0,\eta[;\R^n)}.
\end{aligned}\end{equation*} The constants do not depend on $\eta$. Now, let $\phi$ be in $H^1(\Omega)$. We use the above estimates, the local charts of $\widetilde{\Omega}_{\gamma_0}$ given by \eqref{bound} and a simple change of variables to obtain \eqref{403}.
 \end{proof}

In this second lemma we show that a function in $H^1_0(\Omega)$ can be approached by functions vanishing close to the boundary of $\Omega$. Among other things this lemma is used to give an approximation of $\phi$ via the scale-splitting operator ${\cal Q}_\eps$ (see Lemma \ref{lem26}) and it is also used in the main projection theorem (Theorem \ref{Theorem 2.3 : }).

\begin{lemma}\label{lem83}  Let $\phi$ be in $H^1_0(\Omega)$, there exists $\phi_\eps\in H^1(\R^n)$ satisfying
\begin{equation}\begin{aligned}\label{E100}
&\phi_\eps(x)=0\qquad \hbox{for a.e. } x\not \in \widetilde{\Omega}_{6\sqrt n\eps},\\
&||\phi-\phi_\eps||_{L^2(\Omega)}\le C\eps||\nabla\phi||_{L^2(\Omega;\R^n)},\qquad ||\phi_\eps||_{H^1(\Omega)}\le C||\phi||_{H^1(\Omega)}.
\end{aligned}\end{equation} Moreover, if $\phi\in H^1_{1/\rho}(\Omega)$ then we have
\begin{equation}\label{E200}
\big\|\big(\phi-\phi_\eps\big)/ \rho\big\|_{L^2(\Omega)}\le C\eps||\nabla\phi||_{1/\rho},\qquad ||\phi_\eps||_{1/\rho}\le C||\phi||_{1/\rho}.
\end{equation}
The constant $C$ is independent of $\eps$.
\end{lemma}
\begin{proof}  Let $\phi$ be in $H^1_0(\Omega)$. We define $\phi_\eps$ by
$$\phi_\eps(x)=\left\{
\begin{aligned}
&{(\rho(x)-6\sqrt n \eps)^+\over \rho(x)}\phi(x)\qquad \hbox{for a. e. } \; \; x\in \Omega,\\
& 0 \hskip 4cm  \hbox{for a. e. } \; \; x\in \R^n\setminus \overline{\Omega}.
\end{aligned}\right.$$ where $\delta^+=\max\{0,\delta\}$. The above function $\phi_\eps$ belongs to $H^1(\R^n)$ and satisfies $\phi_\eps=0$ outside $\widetilde{\Omega}_{6\sqrt n\eps}$. Then due to the fact that $\phi/\rho$ belongs to $L^2(\Omega)$ and verifies
$\|{\phi/\rho}\|_{L^2(\Omega)}\le C||\nabla \phi||_{L^2(\Omega;\R^n)}$ we  obtain the estimates in \eqref{E100}. If $\phi\in H^1_{1/\rho}(\Omega)$ we use the estimate \eqref{f0}  to obtain \eqref{E200}.
\end{proof}
\subsection{Reminds and complements on the unfolding operators}\label{RUM}
\noindent In the sequel, we will make use of some definitions and results from \cite{CDG} concerning the periodic unfolding method. Below we remind them briefly.
\vskip 1mm
\subsubsection{Some reminds}\label{Some reminds}

For almost every $x\in \R^n$, there exists an unique
element in $\Z^n$ denoted  $[x]$  such that 
$$x=[x]+\{x\},\qquad \{x\}\in Y.$$ 
$\bullet$ {\it The unfolding operator ${\cal T}_\eps$.}

\noindent For any  $\phi\in  L^1( \Omega)$,  the function ${\cal T}_\eps(\phi)\in L^1(\Omega\times Y)$ is given by
\begin{equation}\label{DefT}
{\cal T}_\eps(\phi)(x,y)=\left\{
\begin{aligned}
\phi\Bigl(\eps\Bigr[{x\over \eps}\Bigr] +\eps y\Bigr)\qquad \hbox{ for a.e. }
(x,y)\in \widehat{\Omega}_\eps\times Y,\\
0 \hskip 3.1cm \hbox{ for a.e. } (x,y)\in \Lambda_\eps\times Y.
\end{aligned}\right.
\end{equation} Since $\Lambda_\eps\subset \widetilde{\Omega}_{\sqrt n \eps}$,  using Proposition 2.5 in \cite{CDG} we get
\begin{equation}\label{Def1}
\Big|\int_\Omega \phi(x)dx-\int_{\Omega\times Y}{\cal T}_\eps(\phi)(x,y)dxdy\Big|\le \int_{\Lambda_\eps} |\phi(x)|dx\le ||\phi||_{L^1(\widetilde{\Omega}_{\sqrt n\eps})}
\end{equation} 
\noindent For $\phi\in L^2(\Omega)$ we have
\begin{equation}\label{T2}
||{\cal T}_\eps(\phi)||_{L^2(\Omega)}\le ||\phi||_{L^2(\Omega)}.
\end{equation} We also have  (see Proposition 2.5 in \cite{CDG}) for $\phi\in H^1(\Omega)$ (resp.  $\psi\in H^1_0(\Omega)$)
\begin{equation}\label{Def2}
\begin{aligned}
&||{\cal T}_\eps(\phi)-\phi||_{L^2(\widehat{\Omega}_\eps\times Y)}\le C\eps||\nabla\phi||_{L^2(\Omega ; \R^n)}\\
\hbox{( resp. }\quad &||{\cal T}_\eps(\psi)-\psi||_{L^2(\Omega\times Y)}\le C\eps ||\nabla\psi||_{L^2(\Omega ; \R^n)}\quad\hbox{).}
\end{aligned}
\end{equation}
$\bullet$ {\it The  local average operator ${\cal M}_\eps$}

\noindent For $\phi\in L^1(\R^n)$, the function ${\cal M}_\eps(\phi)\in L^\infty(\R^n)$  is defined by
\begin{equation}\label{DefM}
{\cal M}_\eps(\phi)(x)=\int_{Y}\phi\Bigl(\eps\Big[{x\over \eps}\Big]+\eps y\Bigr)dy\qquad \hbox{for a.e. } x\in \R^n.
\end{equation} The value of ${\cal M}_\eps(\phi)$ in the cell $\eps(\xi+Y)$ ($\xi\in \Z^n$) will be denoted ${\cal M}_\eps(\phi)(\eps\xi)$. In \cite{CDG} we proved the following results:

\noindent For $\phi\in L^2(\Omega)$  we have
\begin{equation}\label{M2}
||{\cal M}_\eps(\phi)||_{L^2(\Omega)}\le C ||\phi||_{L^2(\Omega)},\qquad 
 ||{\cal M}_\eps(\phi)- \phi||_{H^{-1}(\Omega)}\le C\eps||\phi||_{L^2(\Omega)}
\end{equation}
and for $\psi\in H^1_0(\Omega)$ (resp. $\phi\in H^1(\Omega)$)  we have
\begin{equation}\label{M1}
\begin{aligned}
 ||{\cal M}_\eps(\psi)- \psi||_{L^2(\Omega)}&\le C \eps||\nabla \psi||_{L^2(\Omega;\R^n)}\\
\hbox{(resp. }\quad  ||{\cal M}_\eps(\phi)- \phi||_{L^2(\widehat{\Omega}_\eps)}&\le C \eps||\nabla \phi||_{L^2(\Omega;\R^n)}\hbox{ ).}
\end{aligned}\end{equation} 
\vskip 1mm

\noindent $\bullet$ {\it  The scale-splitting operator ${\cal Q}_\eps$. }

\noindent $\star$  For $\phi\in L^1(\R^n)$, the function ${\cal Q}_\eps(\phi)\in W^{1,\infty}(\R^n)$ is given by
 \begin{equation*}
{\cal Q}_\eps(\phi)(x)=\sum_{\xi\in \Z^n}{\cal M}_\eps(\phi)(\eps\xi)H_{\eps,\xi}(x)\qquad \hbox{for a.e. }x\in\R^n, \\ 
\end{equation*}
where 
\begin{equation*}
\begin{aligned}
& H_{\eps,\xi}(x)=H\Big({x-\eps\xi\over \eps}\Big)\qquad \hbox{with}\\
& H(z)=\left\{
\begin{aligned}
\big(1-|z_1|\big)\big(1-|z_2|\big)\ldots\big(1-|z_n|\big)\quad \hbox{if}\quad z\in [-1,1]^n,\\
0\hskip 3cm\hbox{if }\quad z\in \R^n\setminus [-1,1]^n.
\end{aligned}\right.
\end{aligned}\end{equation*} Below, we remind some results about ${\cal Q}_\eps$  proved in \cite{CDG} and \cite{GG2}.
\vskip 1mm
\noindent $\star$ For $\phi\in L^2(\R^n)$ we have
\begin{equation}\label{QL2}
||{\cal Q}_\eps(\phi)||_{L^2(\R^n)}\le C||\phi||_{L^2(\R^n)},\qquad ||\nabla {\cal Q}_\eps(\phi)||_{L^2(\R^n;\R^n)}\le {C\over \eps}||\phi||_{L^2(\R^n)}
\end{equation} and
 \begin{equation*}
 {\cal Q}_\eps(\phi)\longrightarrow  \phi\qquad\hbox{strongly in }\quad  L^2(\R^n). 
\end{equation*} 
\noindent $\star$ For $\phi\in H^1(\R^n)$  we have
\begin{equation}\label{82}
\begin{aligned}
||\nabla {\cal Q}_\eps(\phi)||_{L^2(\R^n;\R^n)}& \le C||\nabla \phi||_{L^2(\R^n;\R^n)},\\
||\phi-{\cal Q}_\eps(\phi)||_{L^2(\R^n)} & \le C\eps||\nabla\phi||_{L^2(\R^n;\R^n)}
\end{aligned}\end{equation} and 
 \begin{equation}\label{83}
 {\cal Q}_\eps(\phi)\longrightarrow  \phi\qquad\hbox{strongly in }\quad  H^1(\R^n). 
\end{equation}
\noindent $\star$  For $\phi\in L^2(\R^n)$  and $\chi\in L^2(Y)$ we have $\ds {\cal Q}_\eps(\phi)\chi\Bigl(\Bigl\{{\cdot\over\eps}\Bigr\}\Bigr)\in L^2(\R^n)$, $\ds \nabla{\cal Q}_\eps(\phi)\chi\Bigl(\Bigl\{{\cdot\over\eps}\Bigr\}\Bigr)\in L^2(\R^n)$ and
\begin{equation}\label{820}
\begin{aligned}
\big\|{\cal Q}_\eps(\phi)\chi\Bigl(\Bigl\{{\cdot\over\eps}\Bigr\}\Bigr)\big\|_{L^2(\R^n)}
&\leq C\|\phi\|_{L^2(\R^n)}\|\chi\|_{L^2(Y)},\\
 \big\|{\cal Q}_\eps(\phi)\chi\Bigl(\Bigl\{{\cdot\over\eps}\Bigr\}\Bigr)\big\|_{L^2(\widetilde{\Omega}_{\sqrt n \eps})}
&\leq C\|\phi\|_{L^2(\widetilde{\widetilde{\Omega}}_{3\sqrt n \eps})}\|\chi\|_{L^2(Y)}.\\
\end{aligned}\end{equation} Moreover, if $\phi\in H^1(\R^n)$ then we have
\begin{equation}\label{8200}
\begin{aligned}
\big\|\big({\cal Q}_\eps(\phi)-{\cal M}_\eps(\phi)\big)\chi\Bigl(\Bigl\{{\cdot\over\eps}\Bigr\}\Bigr)\big\|_{L^2(\R^n)}
&\leq C\eps\|\nabla\phi\|_{L^2(\R^n;\R^n)}\|\chi\|_{L^2(Y)},\\
\big\|\nabla{\cal Q}_\eps(\phi)\chi\Bigl(\Bigl\{{\cdot\over\eps}\Bigr\}\Bigr)\big\|_{L^2(\R^n;\R^n)}
&\leq C\|\nabla\phi\|_{L^2(\R^n;\R^n)}\|\chi\|_{L^2(Y)},\\
 \big\|\nabla{\cal Q}_\eps(\phi)\chi\Bigl(\Bigl\{{\cdot\over\eps}\Bigr\}\Bigr)\big\|_{L^2(\widetilde{\Omega}_{\sqrt n \eps};\R^n)}
&\leq C\|\nabla\phi\|_{L^2(\widetilde{\widetilde{\Omega}}_{3\sqrt n \eps};\R^n)}\|\chi\|_{L^2(Y)},
\end{aligned}\end{equation} 
\subsubsection{Some complements}
In this subsection, we extend some results given above to functions belonging to $H^1_\rho(\Omega)$. These technical complements intervene in the proofs of the projection theorems and in the Theorem \ref{TH1}.

\begin{lemma}\label{lemNew} For $\phi\in H^1_\rho(\Omega)$ we have
\begin{equation}\label{M21}
\begin{aligned}
&||\rho\big({\cal M}_\eps(\phi)-\phi\big)||_{L^2(\Omega)}\le C\eps||\phi||_\rho,\\
\forall i\in \{1,\ldots,n\},\qquad &||\rho\big(\phi(\cdot+\eps\Ge_i)-\phi\big)||_{L^2(\Omega)}\le C\eps||\phi||_\rho,\\
 &||\rho\big({\cal M}_\eps(\phi)(\cdot+\eps\Ge_i)-{\cal M}_\eps(\phi)\big)||_{L^2(\Omega)}\le C\eps||\phi||_\rho.
\end{aligned}\end{equation}
For $\phi\in L^2_{1/\rho}(\Omega)$ we have
\begin{equation}\label{M20}
 ||{\cal M}_\eps(\phi)-\phi||_{(H^1_\rho(\Omega))^{'}}\le C\eps||\phi/\rho||_{L^2(\Omega)}.
\end{equation}
The constants do not depend on $\eps$.
\end{lemma}
\begin{proof} \noindent  {\it Step 1.}  We prove $\eqref{M21}_1$. 

\noindent Let $\phi$ be in  $ H^1_\rho(\Omega)$ and let $\eps(\xi+Y)$ be a cell included in $\Omega$.

\noindent{\it Case 1:}  $\rho(\eps\xi)\ge 2\sqrt n\eps$. 

\noindent In this case, observing that
$$1\le {\max_{z\in \eps(\xi+Y)}\{\rho(z)\}\over \min_{z\in \eps(\xi+Y)}\{\rho(z)\}}\le 3$$ 
and thanks to the Poincar-Wirtinger's inequality we obtain
$$\begin{aligned}
\int_{\eps(\xi+Y)}[\rho(x)]^2|{\cal M}_\eps(\phi)(\eps\xi)-\phi(x)|^2 dx & \le [\max_{z\in \eps(\xi+Y)}\{\rho(z)\}]^2\int_{\eps(\xi+Y)}|{\cal M}_\eps(\phi)(\eps\xi)-\phi(x)|^2 dx\\
&\le [\max_{z\in \eps(\xi+Y)}\{\rho(z)\}]^2C\eps^2\int_{\eps(\xi+Y)}|\nabla\phi(x)|^2dx\\
&\le C\eps^2\int_{\eps(\xi+Y)}[\rho(x)]^2|\nabla\phi(x)|^2dx.
\end{aligned}$$ 
\noindent{\it Case 2:} $\rho(\eps\xi)\le 2\sqrt n\eps$. 

\noindent In this case we have  
$$\int_{\eps(\xi+Y)}[\rho(x)]^2|{\cal M}_\eps(\phi)(\eps\xi)-\phi(x)|^2 dx \le C\eps^2\int_{\eps(\xi+Y)}|\phi(x)|^2 dx.$$ The cases 1 and 2 lead to
\begin{equation}\label{NlemNew}
\int_{\widehat{\Omega}_\eps}[\rho(x)]^2|{\cal M}_\eps(\phi)(x)-\phi(x)|^2 dx\le C\eps^2\int_{\widehat{\Omega}_\eps}\big([\rho(x)]^2|\nabla\phi(x)|^2+|\phi(x)|^2\big)dx.
\end{equation}
Since $\Lambda_\eps\subset \widetilde{\Omega}_{\sqrt n\eps}$ and due to Lemma \ref{lemP} we get
$$\int_{\Lambda_\eps}[\rho(x)]^2|{\cal M}_\eps(\phi)(x)-\phi(x)|^2 dx \le C\eps^2\int_{\widetilde{\Omega}_{c_0\sqrt n\eps}}|\phi(x)|^2 dx$$ 
 which in turn with  \eqref{NlemNew} gives $\eqref{M21}_1$. Proceeding in the same way we obtain $\eqref{M21}_2$ and $\eqref{M21}_3$. 

\noindent {\it Step 2. } We prove \eqref{M20}.

\noindent Let $\phi$ be in $L^2_{1/\rho}(\Omega)$ and $\psi\in H^1_\rho(\Omega)$. We have
$$\int_{\widehat{\Omega}_\eps}\big({\cal M}_\eps(\phi)-\phi\big)\psi=\int_{\widehat{\Omega}_\eps}\big({\cal M}_\eps(\psi)-\psi\big)\phi.$$ Consequently we obtain
$$\begin{aligned}
\Big|\int_\Omega\big({\cal M}_\eps(\phi)-\phi\big)\psi -\int_{\Omega}\big({\cal M}_\eps(\psi)-\psi\big)\phi\Big|&\le \int_{\Lambda_\eps}\big|\big({\cal M}_\eps(\phi)-\phi\big)\psi\big|+ \int_{\Lambda_\eps}\big|\big({\cal M}_\eps(\psi)-\psi\big)\phi\big|\\
&\le C\big(||\phi||_{L^2(\Lambda_\eps)}+||{\cal M}_\eps(\phi)||_{L^2(\Lambda_\eps)}\big)||\psi||_{L^2(\Omega)}.
\end{aligned}$$ The inclusion $\Lambda_\eps\subset \widetilde{\Omega}_{\sqrt n \eps}$,  the fact that $\phi\in L^2_{1/\rho}(\Omega)$ and the estimates $\eqref{lemPEst}_1$-$\eqref{M21}_1$ lead to
$$\int_\Omega\big({\cal M}_\eps(\phi)-\phi\big)\psi\le C\eps ||\phi/\rho||_{L^2(\Omega)}||\psi||_\rho.$$ Hence \eqref{M20} is proved.
\end{proof} 

\vskip 1mm
\begin{lemma}\label{lem26}  For $\phi\in H^1_\rho(\Omega)$ we have
\begin{equation}\label{RAP}
||\rho\big({\cal Q}_\eps(\phi)-\phi\big)||_{L^2(\Omega)}\le C\eps||\phi||_\rho
\end{equation}
For $\phi\in H^1_{1/\rho}(\Omega)$ and $\phi_\eps$ given by Lemma \ref{lem83} we have
\begin{equation}\label{M22}
\begin{aligned}
&||{\cal Q}_\eps(\phi_\eps)||_{1/\rho}\le C||\phi||_{1/\rho},\qquad 
\big\|\big({\phi-{\cal Q}_\eps(\phi_\eps)\big)/\rho}\big\|_{L^2(\Omega)}\le C\eps||\phi||_{1/\rho},\\
&\forall\Gi =i_1\Ge_1+\ldots+i_n\Ge_n,\qquad (i_1,\ldots, i_n)\in\{0,1\}^{n}\\
&\big\|\big({\cal M}_\eps(\phi_\eps)(\cdot+\eps\Gi)-{\cal M}_\eps(\phi_\eps)\big) /\rho\big\|_{L^2(\Omega)}\le C\eps||\phi||_{1/\rho}.
\end{aligned}\end{equation} For $\phi\in L^2(\R^n)$  and $\chi\in L^2(Y)$
\begin{equation}\label{820RR}
\big\|\big({\cal M}_\eps(\rho\phi)-\rho{\cal M}_\eps(\phi)\big)\chi\Bigl(\Bigl\{{\cdot\over\eps}\Bigr\}\Bigr)\big\|_{L^2(\R^n)}
\leq C\eps\|\phi\|_{L^2(\R^n)}\|\chi\|_{L^2(Y)}.
\end{equation} For $\phi\in H^1_\rho(\Omega)$ and $\chi\in L^2(Y)$
\begin{equation}\label{820RR1}
\begin{aligned}
&\big\|\rho\big({\cal Q}_\eps(\phi)-{\cal M}_\eps(\phi)\big)\chi\Bigl(\Bigl\{{\cdot\over\eps}\Bigr\}\Bigr)\big\|_{L^2(\Omega)}
\leq C\eps\|\phi\|_{\rho}\|\chi\|_{L^2(Y)},\\
&\big\|\rho\nabla{\cal Q}_\eps(\phi)\chi\Bigl(\Bigl\{{\cdot\over\eps}\Bigr\}\Bigr)\big\|_{L^2(\Omega)}
\leq C\|\phi\|_{\rho}\|\chi\|_{L^2(Y)}.
\end{aligned}\end{equation}
The constants do not depend on $\eps$.
\end{lemma}
\begin{proof}  {\it Step 1.} Let $\phi$ be in $H^1_\rho(\Omega)$.  We  first prove 
\begin{equation}\label{RAP2}
||\rho\big({\cal Q}_\eps(\phi)-{\cal M}_\eps(\phi)\big)||_{L^2(\Omega)}\le C\eps||\phi||_\rho.
\end{equation} To do that, we proceed as in the proof of  $\eqref{M21}_1$. Let $\eps(\xi+Y)$ be a cell included in $\Omega$.
\vskip 1mm
\noindent{\it Case 1:}  $\rho(\eps\xi)\ge 3\sqrt n\eps$. 

\noindent In this case we have
$$1\le {\max_{z\in \eps(\xi+Y)}\{\rho(z)\}\over \min_{z\in \eps(\xi+2Y)}\{\rho(z)\}}\le 4\qquad\hbox{and}\qquad  1\le {\max_{z\in \eps(\xi+2Y)}\{\rho(z)\}\over \min_{z\in \eps(\xi+Y)}\{\rho(z)\}}\le {5\over 2}.$$ 
By definition of ${\cal Q}_\eps(\phi)$ we deduce that
$$\begin{aligned}
\int_{\eps(\xi+Y)}[\rho(x)]^2|{\cal Q}_\eps(\phi)(x)-{\cal M}_\eps(\phi)(\eps\xi)|^2 dx & \le [\max_{z\in \eps(\xi+Y)}\{\rho(z)\}]^2\int_{\eps(\xi+Y)}|{\cal Q}_\eps(\phi)(x)-{\cal M}_\eps(\phi)(\eps\xi)|^2 dx\\
&\le [\max_{z\in \eps(\xi+Y)}\{\rho(z)\}]^2 C \eps^2\int_{\eps(\xi+2Y)}|\nabla\phi(x)|^2dx\\
&\le C\eps^2\int_{\eps(\xi+2Y)}[\rho(x)]^2|\nabla\phi(x)|^2dx.
\end{aligned}$$
\noindent{\it Case 2:}  $\rho(\eps\xi)\le 3\sqrt n\eps$. Then again by definition of ${\cal Q}_\eps(\phi)$ we get
$$\int_{\eps(\xi+Y)}[\rho(x)]^2|{\cal Q}_\eps(\phi)(x)-{\cal M}_\eps(\phi)(\eps\xi)|^2 dx \le C\eps^2\int_{\eps(\xi+2Y)}|\phi(x)|^2 dx.$$ As a consequence of both cases  we get
\begin{equation}\label{NN}
\int_{\widehat{\Omega}_\eps}[\rho(x)]^2|{\cal Q}_\eps(\phi)(x)-{\cal M}_\eps(\phi)(x)|^2 dx \le C\eps^2\int_{\Omega}\big([\rho(x)]^2|\nabla\phi(x)|^2+|\phi(x)|^2\big)dx.
\end{equation}
Furthermore we have
$$\int_{\Lambda_\eps}[\rho(x)]^2|{\cal Q}_\eps(\phi)(x)|^2 dx \le C\eps^2\int_{\Lambda_\eps}|{\cal Q}_\eps(\phi)(x)|^2 dx \le C\eps^2\int_{\Omega}|\phi(x)|^2 dx$$ which with \eqref{NN} lead to  \eqref{RAP2}.
Then as a consequence of $\eqref{M21}_1$ and \eqref{RAP2} we get \eqref{RAP}.
\vskip 1mm 
 \noindent  {\it Step 2.}  We  prove $\eqref{M22}_1$. 
 
 \noindent Let  $\phi$ be in $ H^1_{1/\rho}(\Omega)$ and $\phi_\eps$ given by Lemma \ref{lem83}. Due to the fact that  $\phi_\eps(x)=0$ for a.e. $x\in \R^n\setminus \overline{\widetilde{\Omega}}_{6\sqrt n \eps}$, hence  ${\cal Q}_\eps(\phi_\eps)(x)=0$ for every $x\in \Omega$ such that $\rho(x)\le4\sqrt n\eps$. Again we take a cell $\eps(\xi+Y)$ included in $\Omega$ such that $\rho(\eps\xi)\ge 3\sqrt n\eps$. The values taken by ${\cal Q}_\eps(\phi_\eps)$ in the cell $\eps(\xi+Y)$ depend only on the values of $\phi_\eps$ in $\eps(\xi+2Y)$. Then we have
$$\begin{aligned}
& \int_{\eps(\xi+Y)}{1\over [\rho(x)]^2}|\nabla {\cal Q}_\eps(\phi_\eps)(x)|^2dx \le {C\over [\min_{x\in \eps(\xi+Y)}\{\rho(x)\}]^2}\int_{\eps(\xi+2Y)}|\nabla\phi_\eps(x)|^2 dx\\
&\le C{[\max_{x\in \eps(\xi+2Y)}\{\rho(x)\}]^2\over [\min_{x\in \eps(\xi+Y)}\{\rho(x)\}]^2}\int_{\eps(\xi+2Y)}{1\over [\rho(x)]^2}|\nabla\phi_\eps(x)|^2dx\le  C \int_{\eps(\xi+2Y)}{1\over [\rho(x)]^2}|\nabla\phi_\eps(x)|^2dx.
\end{aligned}$$ Adding all these inequalities gives
$$ \int_{\widetilde{\Omega}_{4\sqrt n \eps}}{1\over [\rho(x)]^2}|\nabla {\cal Q}_\eps(\phi_\eps)(x)|^2dx \le  C\int_{\Omega}{1\over [\rho(x)]^2}|\nabla\phi_\eps(x)|^2dx $$
Since  ${\cal Q}_\eps(\phi_\eps)(x)=0$ for every  $x\in \Omega$ such that $\rho(x)\le4\sqrt n\eps$, we get $||{\cal Q}_\eps(\phi_\eps)||_{1/\rho}\le C||\phi_\eps||_{1/\rho}$. We conclude using $\eqref{E200}_2$.
\vskip 1mm
 \noindent  {\it Step 3.}  Now we prove $\eqref{M22}_2$. Again we consider a cell $\eps(\xi+Y)$ included in $\Omega$ such that $\rho(\eps\xi)\ge 3\sqrt n\eps$. We have
$$\begin{aligned}
& \int_{\eps(\xi+Y)}{1\over [\rho(x)]^2}|{\cal Q}_\eps(\phi_\eps)(x)-\phi_\eps(x)|^2 dx\le {C\over [\min_{x\in \eps(\xi+Y)}\{\rho(x)\}]^2}\int_{\eps(\xi+Y)}|{\cal Q}_\eps(\phi_\eps)(x)-\phi_\eps(x)|^2 dx\\
&\le {C\over [\min_{x\in \eps(\xi+Y)}\{\rho(x)\}]^2}\sum_{\Gi\in\{0,1\}^n}\int_{\eps(\xi+\Gi+Y)}|{\cal M}_\eps(\phi_\eps)(\eps\xi+\eps\Gi)-\phi_\eps(x)|^2 dx\\
&\le C\eps^2{[\max_{z\in \eps(\xi+2Y)}\{\rho(z)\}]^2\over [\min_{z\in \eps(\xi+Y)}\{\rho(z)\}]^2}\int_{\eps(\xi+2Y)}{1\over [\rho(x)]^2}|\nabla\phi_\eps(x)|^2dx\le  C \eps^2\int_{\eps(\xi+2Y)}{1\over [\rho(x)]^2}|\nabla\phi_\eps(x)|^2dx.
\end{aligned}$$ Hence we get
$$ \int_{\widetilde{\Omega}_{4\sqrt n \eps}}{1\over [\rho(x)]^2}|{\cal Q}_\eps(\phi_\eps)(x)-\phi_\eps(x)|^2dx \le  C\eps^2\int_{\Omega}{1\over [\rho(x)]^2}|\nabla\phi_\eps(x)|^2dx $$ The above estimate and the fact that  ${\cal Q}_\eps(\phi_\eps)(x)-\phi_\eps(x)=0$ for a.e.  $x\in \Omega$ such that $\rho(x)\le4\sqrt n\eps$ yield $||(\phi_\eps-{\cal Q}_\eps(\phi_\eps))/\rho||_{L^2(\Omega)}\le C\eps||\phi_\eps||_{1/\rho}$. We conclude using  both estimates in \eqref{E200}.

\noindent Proceeding as in the Steps 2 and 3 we obtain $\eqref{M22}_3$, \eqref{820RR} and \eqref{820RR1}.
\end{proof}

\section{Two new projection theorems}\label{TNPT}
\begin{theorem}\label{Theorem 2.2}
Let $\phi$ be in $H^1_{1/\rho}(\Omega)$. There exists  $\widehat{\phi}_\eps\in H^1_{per}(Y ; L^2(\Omega))$ such that
\begin{equation}\label{2.3}
\left\{\begin{aligned}  
&||\widehat{\phi}_\eps||_{H^1(Y ; L^2(\Omega))}\le
C\bigl\{||\phi||_{L^2(\Omega)}+\eps||\nabla \phi||_{[L^2(\Omega)]^n}\bigr\}\\   
&||{\cal T}_\eps(\phi)-\widehat{\phi}_\eps||_{H^1(Y ; (H^1_\rho(\Omega ))^{'})}\le  C\eps\big(||\phi/\rho||_{L^2(\Omega)}+\eps||\phi||_{1/\rho}\big).
\end{aligned}\right.
\end{equation}
\noindent  The constants  depend only on $n$ and  $\partial\Omega$.
\end{theorem}

\begin{proof} Here, we proceed as in the proof of  Proposition 3.3 in \cite {GG1}.  We first reintroduce the open sets $\widehat{\Omega}_{\eps,i}$ and the ''double'' unfolding operators ${\cal T}_{\eps, i}$. We set 
$$\widehat{\Omega}_{\eps,i} =\widehat{\Omega}_{\eps}\cap \big(\widehat{\Omega}_{\eps}-\eps\Ge_i\big),\qquad  K_i=\hbox{interior}\big(\overline{Y}\cup (\Ge_i+\overline{Y})\big),\quad i\in\{1,\ldots,n\}.$$ 
\noindent The unfolding  operator ${\cal T}_{\eps, i}$ from $L^2(\Omega)$
into $L^2(\Omega\times K_i)$ is defined  by
$$\forall\psi\in L^2(\Omega),\qquad {\cal T}_{\eps,
i}(\psi)(x,y)=
\left\{
\begin{aligned}
&\psi\Bigl(\eps\Bigl[{x\over
\eps}\Bigr]_Y+ \eps y\Bigr)\qquad \hbox{for $ x\in\widehat{\Omega}_{\eps,i}$ and  for a.e. $ y\in  K_i$},\\
& 0 \hskip 28mm  \hbox{for $ x\in\Omega\setminus\overline{\widehat{\Omega}}_{\eps,i}$ and for  a.e. $ y\in  K_i$}.
\end{aligned}\right.$$
\noindent  The restriction of  ${\cal T}_{ \eps,i}(\psi)$ to $\widehat{\Omega}_{\eps,i}\times Y$ is equal to ${\cal T}_\eps(\psi)$.
%
%
\vskip 1mm
\noindent{\it Step 1.}  Let us first take $\phi\in L^2_{1/\rho}(\Omega)$. We set $\ds \psi={1\over \rho}\phi$ and we evaluate the difference ${\cal
T}_{\eps, i}(\phi)(.,.. +\Ge_i)-{\cal T}_{\eps, i}(\phi)$ in $L^2(Y ; (H^1_\rho(\Omega))^{'})$.
\smallskip
\noindent For any $\Psi\in H^1_\rho(\Omega)$ a change of variables  gives for a. e. $y\in Y$
\begin{equation*}
\begin{aligned}
 \int_\Omega{\cal T}_{\eps, i}(\phi)(x,y+\Ge_i)\Psi(x)dx&=\int_{
\widehat{\Omega}_{\eps,i}}{\cal T}_{\eps}(\phi)(x+\eps\Ge_i,y)\Psi(x)dx\\ 
&= \int_{\widehat{\Omega}_{\eps,i}+\eps\Ge_i}{\cal T}_{\eps}(\phi)(x,y)\Psi(x-\eps\Ge_i)dx.
\end{aligned}\end{equation*} Then we obtain  for a. e. $y\in Y$
\begin{equation*}
\begin{aligned}
  &\Bigl|\int_{\Omega}\bigl\{{\cal T}_{\eps, i}(\phi)(., y+\Ge_i)-{\cal T}_{\eps, i}(\phi)(.,
y)\bigr\}\Psi -\int_{\widehat{\Omega}_{\eps,i}}{\cal T}_{\eps}(\psi)(., y)\rho\bigl\{\Psi(.-\eps\Ge_i)-\Psi\bigr\}
\Bigr|\\ 
\le & \Big|\int_{\widehat{\Omega}_{\eps,i}}{\cal T}_{\eps}(\psi)(., y)\big({\cal T}_\eps(\rho)-\rho\big)\bigl\{\Psi(.-\eps\Ge_i)-\Psi\bigr\}
\Big|+C||{\cal T}_{\eps}(\phi)(.,y)||_{L^2(\widetilde{\Omega}_{2\sqrt n\eps})}||\Psi||_{
L^2(\widetilde{\Omega}_{2\sqrt n\eps})}.
\end{aligned}\end{equation*}  Estimate $\eqref{M21}_2$ leads to
\begin{equation*}
 ||\rho\big(\Psi(.-\eps\Ge_i)-\Psi\big)||_{L^2(\widehat{\Omega}_{\eps,i})}\le C\eps ||\Psi||_\rho\qquad \forall i\in\{1,\ldots,n\}.
\end{equation*}  We have 
\begin{equation}\label{Trho}
||{\cal T}_\eps(\rho)-\rho||_{L^\infty(\Omega)}\le C\eps.
\end{equation} The above inequalities imply
\begin{equation*}
\begin{aligned}
 <{\cal T}_{\eps, i}&(\phi)(., y+\Ge_i)-{\cal T}_{\eps, i}(\phi)(., y)\,,\, \Psi>_{(H^1_\rho(\Omega))^{'},H^1_\rho( \Omega)}\\   
=&\int_{\Omega}\bigl\{{\cal T}_{\eps, i}(\phi)(x, y+\Ge_i)-{\cal T}_{\eps, i}(\phi)(x, y)\bigr\}\Psi(x)dx\\ 
\le  &C\eps ||\Psi||_\rho\|{\cal T}_{\eps}(\psi)(.,y)\|_{L^2(\Omega )}+C\eps ||\Psi||_{L^2(\Omega)}\|{\cal T}_{\eps}(\psi)(.,y)\|_{L^2(\Omega )}\\
+ &C||{\cal T}_{\eps}(\phi)(.,y)||_{L^2(\widetilde{\Omega}_{2\sqrt n\eps})}||\Psi||_{L^2(\widetilde{\Omega}_{2\sqrt n\eps})}.
\end{aligned}\end{equation*}  Therefore, for a.e. $y\in Y$ we have
$$||{\cal T}_{\eps, i}(\phi)(.,y +\Ge_i)-{\cal T}_{\eps, i}(\phi)(., y)||_{(H^1_\rho(\Omega))^{'}}\le  C
\eps\|{\cal T}_{\eps}(\psi)(.,y)\|_{L^2( \Omega)}+  C ||{\cal T}_{\eps}(\phi)(.,y)
||_{L^2(\widetilde{\Omega}_{2\sqrt n \eps})}$$
\noindent which leads to the following estimate of the difference between ${\cal T}_{\eps, i}(\phi)_{|_{\Omega\times
Y}}$ and one of its translated :
\begin{equation}\label{824}
\begin{aligned}
||{\cal T}_{\eps, i}(\phi)(.,.. +\Ge_i)-{\cal T}_{\eps, i}(\phi)||_{L^2(Y ; (H^1_\rho(\Omega))^{'})}
&\le   C\eps||\phi/\rho||_{L^2(\Omega)}+  C||\phi||_{L^2(\widetilde{\Omega}_{2\sqrt n\eps})}\\
&\le   C\eps||\phi/\rho||_{L^2(\Omega)}.
\end{aligned}\end{equation} The constant depends only on the boundary of  $\Omega$.
\vskip 1mm
\noindent{\it Step 2. } Let $\phi\in H^1_{1/\rho}(\Omega)$. The above estimate \eqref{824} applied to $\phi$ and its partial derivatives give
\begin{equation*}
\begin{aligned}  
||{\cal T}_{\eps, i}(\phi)(., .. +\Ge_i)-{\cal T}_{\eps, i}(\phi) ||_{ L^2(Y ; (H^1_\rho(\Omega))^{'})}
& \le C\eps||\phi/\rho||_{L^2(\Omega)}\\  
||{\cal T}_{\eps, i}(\nabla \phi)(., .. +\Ge_i)-{\cal T}_{\eps, i}(\nabla  \phi)||_{ [L^2(Y ; (H^1_\rho(\Omega))^{'}]^n)} &\le C\eps ||\phi||_{1/\rho}.
\end{aligned}\end{equation*}  which in turn lead to  (we recall  that  $\nabla_y\bigl({\cal T}_{\eps, i}(\phi)\big)=\eps{\cal T}_{\eps, i}(\nabla \phi)$). 
\begin{equation*}\begin{aligned} 
||{\cal T}_{\eps,i}(\phi)(., .. +\Ge_i)-{\cal T}_{\eps,i}(\phi)||_{H^1(Y ; (H^1_\rho(\Omega))^{'})} &\le
C\eps\big(||\phi/\rho||_{L^2(\Omega)}+\eps|| \phi||_{1/\rho}\big).
\end{aligned}\end{equation*} 
\noindent From these inequalities for $i\in\{1,\ldots,n\}$  we deduce the estimate of the difference of the traces of the function $y\longrightarrow {\cal T}_{\eps}(\phi)(.,y)$ on the faces $Y_i\doteq\{y\in \overline{Y}\; |\; y_i=0\}$ and $\Ge_i+Y_i$ 
\begin{equation}\label{825}
\begin{aligned} 
||{\cal T}_{\eps}(\phi)(.,.. +\Ge_i)-{\cal T}_{\eps}(\phi)||_{H^{1/ 2}(Y_i  ; (H^1_\rho(\Omega))^{'})}\le C\eps\big(||\phi/\rho||_{L^2(\Omega)}+\eps||\phi||_{1/\rho}\big).
\end{aligned}
\end{equation} These estimates ($i\in\{1,\ldots,n\}$) give a  measure of the periodic defect of the function 
 $y\longrightarrow{\cal T}_\eps(\phi)(.,y)$ (see \cite {GG1}). 
 
 \noindent Then we decompose  ${\cal T}_{\eps}(\phi)$ into the sum of an element
belonging to $H^1_{per}(Y ; L^2(\Omega ))$ and one to  $\bigl(H^1(Y ; L^2(\Omega))\big)^{\perp}$ (the
orthogonal of $H^1_{per}(Y ; L^2(\Omega))$ in $H^1(Y ; L^2(\Omega))$, see \cite {GG1})
\begin{equation}\label{826}
{\cal T}_{\eps}(\phi)=\widehat{\phi}_\eps+\overline{\phi}_\eps,\qquad\widehat{\phi}_\eps
\in H^1_{per}(Y ; L^2(\Omega )),\qquad \overline{\phi}_\eps\in  \bigl(H^1(Y ; L^2(\Omega))\big)^\perp.
\end{equation} The function $y\longrightarrow {\cal T}_\eps(\phi)(., y)$ takes its values in a finite dimensional  space, 
$$\overline{\phi}_\eps(., ..) =\sum_{\xi\in \Xi_\eps}\overline{\phi}_{\eps, \xi}(..)\chi_{\eps, \xi}(.)$$ where
$\chi_{\eps, \xi}(.)$ is the  characteristic  function of the cell $\eps(\xi+Y)$ and where $\overline{\phi}_{\eps, \xi}(..)\in
\bigl(H^1(Y)\big)^\perp$  (the orthogonal of $H^1_{per}(Y)$ in $H^1(Y)$, see \cite {GG1}). The   decomposition \eqref{826} is the same in $H^1(Y ; (H^1_\rho(\Omega))^{'})$ and  we have
\begin{equation*}
\begin{aligned} 
||\widehat{\phi}_\eps||^2_{H^1(Y ; L^2(\Omega))}+||\overline{\phi}_\eps||^2_{H^1(Y ; L^2(\Omega))}
=  ||{\cal T}_\eps(\phi)||^2_{H^1(Y ; L^2(\Omega))} 
\le  C\bigl\{||\phi||_{L^2(\Omega)}+\eps||\nabla \phi||_{[L^2(\Omega)]^n}\bigr\}^2.
\end{aligned}\end{equation*} It gives  the first inequality in \eqref{2.3} and the estimate of $\overline{\phi}_\eps$ in
$H^1(Y ; L^2(\Omega))$. From Theorem 2.2 in \cite {GG1} and  \eqref{825} we obtain a finer estimate of $\overline{\phi}_\eps$ in $H^1(Y ; (H^1_\rho(\Omega))^{'})$
$$||\overline{\phi}_\eps||_{H^1(Y ; (H^1_\rho(\Omega))^{'})} \le C\eps\big(||\phi/\rho||_{L^2(\Omega)}+\eps|| \phi||_{1/\rho}\big).$$ It is the second inequality in \eqref{2.3}.
\end{proof}
\begin{theorem}\label{Theorem 2.3 : } For  $\phi\in H^1_{1/\rho}(\Omega)$, there exists
$\widehat{\phi}_\eps\in H^1_{per}(Y ; L^2(\Omega))$ such that
\begin{equation}\label{8207}
\begin{aligned} 
&||\widehat{\phi}_\eps||_{H^1(Y ; L^2(\Omega))}\le C||\nabla \phi||_{[L^2(\Omega)]^n},\\  &||{\cal
T}_\eps(\nabla \phi)-\nabla \phi-\nabla_y\widehat{\phi}_\eps||_{ [L^2(Y ; (H^1_\rho(\Omega))^{'})]^n}
\le C\eps|| \phi||_{1/\rho}.
\end{aligned}\end{equation}
\noindent  The constants  depend only on  $\partial\Omega$.
\end{theorem}

\begin{proof} Let $\phi$ be in $H^1_{1/\rho}(\Omega)$ and $\psi=\phi/\rho\in H^1_0(\Omega)$. The function $\phi$ is extended by $0$ outside of $\Omega$. We decompose $\phi$ as 
$$\phi=\Phi+\eps \underline{\phi} ,
\quad\hbox{where}\enskip \Phi={\cal Q}_\eps(\phi_\eps)\quad\hbox{and} \quad
\underline{\phi}={1\over\eps}\Big(\phi-{\cal Q}_\eps(\phi_\eps)\Big)$$ where $\phi_\eps$ is given by Lemma \ref{lem83}. We have  $\Phi$ and $\underline{\phi}\in H^1_0(\Omega)$ and due to \eqref{M22} we get the following estimates:
\begin{equation}\label{82100}
||\Phi||_{1/\rho}+\eps|| \underline{\phi}||_{1/\rho}+||\underline{\phi}/\rho||_{L^2(\Omega)}\le  C||\phi||_{1/\rho}.
\end{equation}
\noindent The projection Theorem  \ref{Theorem 2.2} applied to $\underline{\phi}\in H^1_{1/\rho}(\Omega)$ gives an element $\widehat{\phi}_{\eps}$ in
$H^1_{per}(Y ; L^2(\Omega))$ such that
\begin{equation}\label{390}
\begin{aligned} 
 & ||\widehat{\phi}_{\eps}||_{H^1(Y ; L^2(\Omega))}\le C||\phi||_{1/\rho},\\
 & ||{\cal T}_\eps(\underline{\phi})-\widehat{\phi}_{\eps}||_{H^1(Y ;  (H^1_\rho(\Omega))^{'})}\le C\eps||\phi||_{1/\rho}.
\end{aligned}\end{equation}
\noindent   Now we  evaluate $||{\cal T}_\eps(\nabla \Phi)-\nabla \Phi||_{[L^2(Y ; (H^1_\rho(\Omega))^{'})]^n}$. 
\vskip 1mm
\noindent From \eqref{M20}, $\eqref{M22}_1$ and \eqref{82100} we get
\begin{equation}\label{8210}
\|\nabla \Phi-{\cal M}_\eps(\nabla\Phi)\|_{(H^1_\rho(\Omega;\R^n))^{'}}\le C\eps\|\phi\|_{1/\rho}.
\end{equation}
We set 
$$
\begin{aligned}
H^{(1)}(z)&=\left\{
\begin{aligned}
\big(1-|z_2|\big)(1-|z_3|\big)\ldots\big(1-|z_n|\big)\quad \hbox{if}\quad z=(z_1,z_2,\ldots, z_n)\in [-1,1]^{n},\\
0\hskip 3cm\hbox{if }\quad z\in \R^{n}\setminus [-1,1]^{n}.
\end{aligned}\right.\\
\GI&=\Big\{\Gi\;|\; \Gi =i_2\Ge_2+\ldots+i_n\Ge_n,\quad (i_2,\ldots, i_n)\in\{0,1\}^{n-1}\Big\}
\end{aligned}$$
For $\xi\in \Z^n$ and for every $(x,y)\in \eps(\xi+Y)\times Y$ we have
$$
\begin{aligned}
{\cal T}_\eps \Bigl({\partial\Phi\over \partial x_1}\Bigr)(x, y)&=\sum_{\Gi\in \GI}{{\cal M}_\eps(\phi_\eps)\big(\eps(\xi+\Ge_1+\Gi)\big)-{\cal M}_\eps(\phi_\eps)\big(\eps(\xi+\Gi)\big)\over \eps} H^{(1)}(y-\Gi)\\
{\cal  M}_\eps \Bigl({\partial\Phi\over \partial x_1}\Bigr)(\eps\xi)&={1\over 2^{n-1}}\sum_{\Gi\in \GI}{{\cal M}_\eps(\phi_\eps)\big(\eps(\xi+\Ge_1+\Gi)\big)-{\cal M}_\eps(\phi_\eps)\big(\eps(\xi+\Gi)\big)\over \eps} .
\end{aligned}
$$
Now, let us take  $\psi\in H^1_\rho(\Omega)$. We recall that $\phi_\eps(x)=0$ for a.e. $x\in \R^n\setminus \overline{\widetilde{\Omega}}_{6\sqrt n \eps}$, hence  $\Phi(x)=0$ for  $x\in \R^n\setminus \overline{\widetilde{\Omega}}_{3\sqrt n \eps}$; as a first consequence   $\ds {\cal M}_\eps\Bigl({\partial\Phi\over\partial x_1}\Bigr)=0$ in $\Lambda_\eps$. 

\noindent For  $y\in Y$ we have 
\begin{equation*}\begin{aligned} 
<{\cal T}_\eps \Bigl({\partial\Phi\over \partial x_1}\Bigr)(., y)-{\cal M}_\eps\Bigl({\partial\Phi\over\partial
x_1}\Bigr),\psi>_{(H^1_\rho(\Omega))^{'}, H^1_\rho(\Omega)}&=
\int_\Omega\Bigl\{{\cal T}_\eps\Bigl({\partial\Phi\over \partial x_1}\Bigr)(x , y)-{\cal M}_\eps\Bigl({\partial\Phi\over
\partial x_1}\Bigr)(x)\Bigr\}\psi(x)dx\\ 
&=\int_{\widehat{\Omega}_\eps}\Bigl\{{\cal T}_\eps \Bigl({\partial\Phi\over \partial x_1}\Bigr)(x, y)-{\cal M}_\eps\Bigl({\partial\Phi\over
\partial x_1}\Bigr)(x)\Bigr\}{\cal M}_\eps(\psi)(x)dx.
\end{aligned}
\end{equation*} 
Besides we have
\begin{equation*}
\begin{aligned}
\int_{\widehat{\Omega}_\eps}{\cal M}_\eps\Bigl({\partial\Phi\over\partial x_1}\Bigr)(x){\cal M}_\eps(\psi)(x)dx&=\eps^{n}\sum_{\xi\in \Z^n}{\cal  M}_\eps \Bigl({\partial\Phi\over \partial x_1}\Bigr)(\eps\xi){\cal M}_\eps(\psi)(\eps \xi)\\
&={\eps^{n}\over 2^{n-1}}\sum_{\xi\in \Z^n}\sum_{\Gi\in \GI}{{\cal M}_\eps(\phi_\eps)\big(\eps(\xi+\Ge_1+\Gi)\big)-{\cal M}_\eps(\phi_\eps)\big(\eps(\xi+\Gi)\big)\over \eps}{\cal M}_\eps(\psi)(\eps \xi)\\
&={\eps^{n}\over 2^{n-1}}\sum_{\xi\in \Z^n}\sum_{\Gi\in \GI}{{\cal M}_\eps(\psi)\big(\eps(\xi-\Ge_1)\big)-{\cal M}_\eps(\psi)\big(\eps\xi\big)\over \eps}{\cal M}_\eps(\phi_\eps)(\eps( \xi+\Gi))
\end{aligned}\end{equation*} and
\begin{equation*}
\begin{aligned}
&\int_{\widehat{\Omega}_\eps}{\cal T}_\eps\Bigl({\partial\Phi\over\partial x_1}\Bigr)(x, y){\cal M}_\eps(\psi)(x)dx\\
=&\eps^{n}\sum_{\xi\in \Z^n}\sum_{\Gi\in \GI}\Big[{{\cal M}_\eps(\phi_\eps)\big(\eps(\xi+\Ge_1+\Gi)\big)-{\cal M}_\eps(\phi_\eps)\big(\eps(\xi+\Gi)\big)\over \eps}\Big] H^{(1)}(y-\Gi){\cal M}_\eps(\psi)(\eps \xi)\\
=&\eps^{n}\sum_{\xi\in \Z^n}\sum_{\Gi\in \GI}{{\cal M}_\eps(\psi)\big(\eps(\xi-\Ge_1)\big)-{\cal M}_\eps(\psi)\big(\eps\xi)\big)\over \eps}H^{(1)}(y-\Gi){\cal M}_\eps(\phi_\eps)(\eps (\xi+\Gi))
\end{aligned}\end{equation*} Due to the fact that  $\phi_\eps(x)=0$ for a.e. $x\in \R^n\setminus \overline{\widetilde{\Omega}}_{6\sqrt n \eps}$, in the above summations we only  take the $\xi$'s belonging to  $\Xi_\eps$ and satisfying $\rho(\eps \xi)\ge 3\sqrt n\eps$. Hence
\begin{equation*}\begin{aligned} 
&<{\cal T}_\eps \Bigl({\partial\Phi\over \partial x_1}\Bigr)(., y)-{\cal M}_\eps\Bigl({\partial\Phi\over\partial
x_1}\Bigr),\psi>_{(H^1_\rho(\Omega))^{'}, H^1_\rho(\Omega)}\\
=&\eps^{n}\sum_{\xi\in \Z^n}{{\cal M}_\eps(\psi)\big(\eps(\xi-\Ge_1)\big)-{\cal M}_\eps(\psi)\big(\eps\xi)\big)\over \eps}\sum_{\Gi\in \GI}\big[H^{(1)}(y-\Gi)-{1\over 2^{n-1}}\big]{\cal M}_\eps(\phi_\eps)(\eps (\xi+\Gi)).
\end{aligned}\end{equation*} Thanks to the identity relation  $\ds \sum_{\Gi\in \GI}\big[H^{(1)}(y-\Gi)-{1\over 2^{n-1}}\big]=0$ we obtain that
$$\Big|\sum_{\Gi\in \GI}\big[H^{(1)}(y-\Gi)-{1\over 2^{n-1}}\big]{\cal M}_\eps(\phi_\eps)(\eps (\xi+\Gi))\Big|\le \sum_{\Gi\in \GI}\big|{\cal M}_\eps(\phi_\eps)(\eps (\xi+\Gi))-{\cal M}_\eps(\phi_\eps)(\eps\xi)\big|.$$ 
Taking into account the last  equality and inequality above we deduce that
\begin{equation*}\begin{aligned} 
&<{\cal T}_\eps \Bigl({\partial\Phi\over \partial x_1}\Bigr)(., y)-{\cal M}_\eps\Bigl({\partial\Phi\over\partial
x_1}\Bigr),\psi>_{(H^1_\rho(\Omega))^{'}, H^1_\rho(\Omega)}\\
=&\eps^{n}\sum_{\xi\in \Z^n}\sum_{\Gi\in \GI}\Big|{{\cal M}_\eps(\psi)\big(\eps(\xi-\Ge_1)\big)-{\cal M}_\eps(\psi)\big(\eps\xi)\big)\over \eps}\Big|\big|{\cal M}_\eps(\phi_\eps)(\eps (\xi+\Gi))-{\cal M}_\eps(\phi_\eps)(\eps\xi)\big|\\
=&{1\over \eps}\sum_{\Gi\in \GI}\int_\Omega\big|{\cal M}_\eps(\psi)(\cdot-\eps\Ge_1)-{\cal M}_\eps(\psi)\big|\,\big|{\cal M}_\eps(\phi_\eps)(\cdot+\eps\Gi)-{\cal M}_\eps(\phi_\eps)\big|\\
\le&{C\over \eps}\sum_{\Gi\in \GI}\big\|\rho\big({\cal M}_\eps(\psi)(\cdot-\eps\Ge_1)-{\cal M}_\eps(\psi)\big)\big\|_{L^2(\Omega)}\Big\|{1\over \rho}\big({\cal M}_\eps(\phi_\eps)(\cdot+\eps\Gi)-{\cal M}_\eps(\phi_\eps)\big)\Big\|_{L^2(\Omega)}.
\end{aligned}\end{equation*} Due to $\eqref{M21}_3$ and $\eqref{M22}_3$ we finally get
$$
<{\cal T}_\eps \Bigl({\partial\Phi\over \partial x_1}\Bigr)(., y)-{\cal M}_\eps\Bigl({\partial\Phi\over\partial
x_1}\Bigr),\psi>_{(H^1_\rho(\Omega))^{'}, H^1_\rho(\Omega)}\le C\eps ||\phi_\eps||_{1/\rho}||\psi||_\rho.
$$ It leads to
\begin{equation}\label{39}
\Big\|{\cal T}_\eps \Bigl({\partial\Phi\over \partial x_1}\Bigr)-{\cal M}_\eps\Bigl({\partial\Phi\over\partial
x_1}\Bigr)\Big\|_{L^\infty(Y;(H^1_\rho(\Omega))^{'})}\le C\eps ||\phi_\eps||_{1/\rho}.
\end{equation}
\noindent Besides we have
\begin{equation*}\begin{aligned}
\int_\Omega {\partial \underline{\phi}\over \partial x_1}(x)\psi(x)dx=-\int_\Omega\underline{\phi}(x){\partial \psi\over \partial x_1}(x)dx\le
C||\underline{\phi}/\rho||_{L^2(\Omega)}||\psi||_{\rho}\le C||\phi||_{1/\rho}||\psi||_{\rho}.
\end{aligned}\end{equation*} Hence
$\ds \Big\|\eps{\partial \underline{\phi}\over \partial x_1}\Big\|_{(H^1_\rho(\Omega;\R^n))^{'}}\le  C\eps||\phi||_{1/\rho}$. This last estimate with  $\eqref{E200}_2$, \eqref{8210} and \eqref{39} yield 
$$\Big\|{\cal T}_\eps \Bigl({\partial\Phi\over \partial x_1}\Bigr)-{\partial\phi\over\partial
x_1}\Big\|_{L^\infty(Y;(H^1_\rho(\Omega))^{'})}\le C\eps ||\phi_\eps||_{1/\rho}.$$ In the same way  we prove the  estimates for the partial derivatives of $\Phi$ with respect to $x_i$, $i\in\{2,\ldots,n\}$.  Hence we get $\ds \|{\cal T}_\eps(\nabla \Phi)-\nabla\phi\|_{[L^\infty(Y;(H^1_\rho(\Omega))^{'})]^n}\le C\eps ||\phi_\eps||_{1/\rho}$. Then   thanks to \eqref{390} the second estimate in \eqref{8207} is proved.
\end{proof}
\section{Reminds about the classical periodic homogenization problem}\label{Rhom}
\noindent We consider the homogenization problem
\begin{equation}\label{Pbeps}
\phi^{\eps}\in H^1_0(\Omega), \qquad \int_\Omega A_\eps(x)\nabla\phi^{\eps}(x)\nabla\psi(x)dx=
\int_{\Omega}f(x)\psi(x)dx, \qquad \forall \psi\in H^1_0(\Omega),
\end{equation}  where 

$\bullet$ $\ds A_\eps (x)=A\Big(\Big\{{x\over \eps}\Big\}\Big)$ for a.e. $x\in \Omega$, where $A$ is a square matrix  belonging to
$L^\infty(Y;\R^{n\times n})$ and satisfying the condition of uniform  ellipticity  $c|\xi|^2\le A(y)\xi\cdot\xi $  for a.e. $y\in Y$, with $c$ a strictly positive constant,
 
$\bullet$ $f\in L^2(\Omega)$.
\vskip 2mm
\noindent We showed in \cite{CDG} that 
$${\cal T}_\eps(\nabla \phi^\eps)\longrightarrow \nabla\Phi+\nabla_y\widehat{\phi}\quad \hbox{strongly in}\quad L^2(\Omega\times Y;\R^n)$$ where  $(\Phi,\widehat{\phi})\in  H^1_0(\Omega)\times L^{2}(\Omega ; H^1_{per}(Y))$ is the solution of the 
problem of unfolding homogenization
 \begin{equation*}
\begin{aligned}
 &\forall (\Psi,\widehat{\psi})\in H^1_0(\Omega)\times L^{2}(\Omega ; H^1_{per}(Y)) \\ 
 & \int_{\Omega}\int_{Y}A(y)\bigl\{\nabla \Phi(x)+\nabla_y\widehat{\phi}(x,y)\bigr\}\,
\bigl\{\nabla \Psi(x)+\nabla_y\widehat{\psi}(x,y)\bigr\}dxdy=\int_\Omega f(x)\Psi(x)dx.
\end{aligned}\end{equation*}

\noindent   The correctors $\chi_i$, $i\in\{1,\ldots,n\}$, are the solutions of the  variational problems 
\begin{equation}\label{cor}
\begin{aligned}
&\chi_i\in H^1_{per}(Y),\qquad\int_Y \chi_i=0,\\
&  \int _YA(y)\nabla_y\bigl(\chi_i(y)+y_i\big)\nabla_y\psi(y)dy=0,\qquad
\forall\psi\in H^1_{per}(Y).
\end{aligned}\end{equation}
\noindent They allow  to express $\widehat{\phi}$ in terms of the partial derivatives of $\Phi$ 
\begin{equation}\label{phihat}
\widehat{\phi}=\sum_{i=1}^n{\partial\Phi\over\partial x_i}\chi_i
\end{equation}
 and to give the  homogenized problem satisfied by $\Phi$
\begin{equation}\label{Homo}
\Phi\in H^1_0(\Omega),\qquad \int_\Omega{\cal A}\nabla \Phi(x)\nabla \Psi(x)dx=\int_\Omega f(x)\Psi(x)dx,\qquad \forall \Psi\in H^1_0(\Omega)
\end{equation}
 where (see \cite{CDG})
\begin{equation}\label{Aij}
{\cal A}_{ij}=\sum_{k,l=1}^n\int_Y  a_{kl}(y){\partial(y_j+\chi_j(y))\over \partial y_l}{\partial(y_i+\chi_i(y))\over\partial y_k}dy.
\end{equation}
\section{An operator from  $H^{-1/2}(\partial \Omega)$ into $L^2(\Omega)$}\label{Op}
  
{\it From now on,  $\Omega$ is a bounded domain with a ${\cal C}^{1,1}$ boundary or an open bounded convex set.}
\vskip 1mm
In this section we first  introduce a lifting operator $\GT$ (defined by \eqref{500}) from $H^{1/2}(\partial\Omega)$ into $H^1(\Omega)$. This operator and the estimate \eqref{phifinit} are in fact sufficient to obtain the error estimates with a non-homogeneous Dirichlet condition (Theorem \ref{TH3}); one of the aim of this paper. Then we extend this operator. The extension of $\GT$  from $H^{-1/2}(\partial\Omega)$ into $H^1_\rho(\Omega)$ is essential in order to get a sharper estimate \eqref{Estphi2} than $\eqref{Estphi}_1$. In Theorem \ref{TH65} we give an application based on \eqref{Estphi2}, in this theorem we investigate a first case of strongly oscillating boundary data.
\vskip 2mm
Let $g$ be  in $H^{1/2}(\partial \Omega)$, there exists one $\phi_g\in H^1(\Omega)$ such that 
\begin{equation}\label{500}
\hbox{div}({\cal A}\nabla\phi_g)=0\qquad \hbox{in}\quad \Omega,\qquad \phi_g=g\qquad \hbox{on}\quad \partial \Omega
\end{equation} where  ${\cal A}$ is the matrix given by \eqref{Aij}. We have
\begin{equation}\label{phifinit}
||\phi_g||_{H^1(\Omega)}\le C||g||_{H^{1/2}(\partial\Omega)}.
\end{equation} We denote by $\GT$  the  operator from $H^{1/2}(\partial \Omega)$ into $H^1(\Omega)$ which associates to  $g\in H^{1/2}(\partial \Omega)$ the function $\phi_g\in H^1(\Omega)$.
\vskip 1mm
Now, let   $(\psi, \Psi)$ be a couple in $[{\cal C}^\infty(\overline{\Omega})]^2$,  integrating by parts over $\Omega$ gives
\begin{equation*}
\int_\Omega{\cal A}\nabla\psi(x)\nabla \Psi(x)dx=-\int_\Omega \psi(x) \hbox{div}({\cal A}^T\nabla\Psi)(x)dx+\int_{\partial\Omega}\psi(x)( {\cal A}^T\nabla\Psi)(x)dx\cdot\nu(x) d\sigma.
\end{equation*} The space ${\cal C}^\infty(\overline{\Omega})$ being dense in $H^1(\Omega)$ and $H^2(\Omega)$, hence the above equality holds true for any  $\psi\in H^1(\Omega)$ and any $\Psi\in H^2(\Omega)$.  Hence, for   $\Psi\in H^1_0(\Omega)\cap H^2(\Omega)$ and $\phi_g$ defined by \eqref{500} we get
\begin{equation}\label{phif}
\int_\Omega \phi_g(x)\hbox{div}({\cal A}^T\nabla\Psi)(x)dx=\int_{\partial\Omega}g(x)\,({\cal A}^T\nabla\Psi)(x)\cdot\nu(x) d\sigma.
\end{equation} 
\vskip 1mm
Under the assumption  on $\Omega$   the function $\Psi(g)$ defined by
 $$\Psi(g)\in H^1_0(\Omega),\qquad  \hbox{div}({\cal A}^T\nabla\Psi(g))=\phi_g\qquad \hbox{in}\quad \Omega$$ belongs to $H^1_0(\Omega)\cap H^2(\Omega)$ and 
satisfies
$$
||\Psi(g)||_{H^2(\Omega)}\le C||\phi_g||_{L^2(\Omega)}.$$ Taking $\Psi=\Psi(g)$ in the above equality \eqref{phif} we obtain
\begin{equation*}
\begin{aligned}
\int_\Omega |\phi_g(x)|^2dx=\int_{\partial\Omega}g(x)\,( {\cal A}^T\nabla\Psi(g)(x))\cdot\nu(x) d\sigma & \le ||g||_{H^{-1/2}(\partial\Omega)}||( {\cal A}^T\nabla\Psi(g))\cdot\nu||_{H^{1/2}(\partial\Omega)}\\
& \le C ||g||_{H^{-1/2}(\partial\Omega)}||\Psi(g)||_{H^2(\Omega)}.
\end{aligned}\end{equation*} This leads to
\begin{equation}\label{T}
||\phi_g||_{L^2(\Omega)}\le C ||g||_{H^{-1/2}(\partial\Omega)}.
\end{equation} 
\vskip 1mm
 Due to \eqref{T}, the operator $\GT$ admits an extension (still denoted $\GT$) from  $H^{-1/2}(\partial \Omega)$ into $L^2(\Omega)$ and we have
\begin{equation*}
\forall g\in H^{-1/2}(\partial\Omega),\qquad ||\GT(g)||_{L^2(\Omega)}\le C ||g||_{H^{-1/2}(\partial\Omega)}.
\end{equation*} For $g\in H^{-1/2}(\partial\Omega)$, we also denote $\phi_g=\GT(g)$. This function  is the ''very weak'' solution of the problem
 $$\phi_g\in L^2(\Omega),\qquad \hbox{div}({\cal A}\nabla\phi_g)=0\qquad \hbox{in}\quad \Omega,\qquad \phi_g=g\qquad \hbox{on}\quad \partial \Omega$$
 or  the solution of the following:
\begin{equation}\label{PBphig}
\begin{aligned}
&\phi_g\in L^2(\Omega),\\
& \int_\Omega \phi_g(x)\, \hbox{div}({\cal A}^T\nabla\psi(x))dx=<g, ( {\cal A}^T\nabla\psi)\cdot\nu>_{H^{-1/2}(\partial\Omega) , H^{1/2}(\partial\Omega)},\\
&\forall \psi\in H^1_0(\Omega)\cap H^2(\Omega).
\end{aligned}\end{equation}
\begin{lemma}\label{GT}
The operator $\GT$ is a bicontinuous linear operator from $H^{-1/2}(\partial\Omega)$ onto
$$\GH=\Bigl\{\phi\in L^2(\Omega)\;\; |\;\;  \hbox{div}({\cal A}\nabla\phi)=0\quad \hbox{in}\quad \Omega\Big\}.$$ There exists a  constant $C\ge 1$ such that 
\begin{equation}\label{30}
\forall g\in H^{-1/2}(\partial\Omega),\qquad {1\over C}||g||_{H^{-1/2}(\partial\Omega)}\le ||\GT(g)||_{L^2(\Omega)}\le C||g||_{H^{-1/2}(\partial\Omega)}.
\end{equation}
\end{lemma}
\begin{proof} Let $\phi$ be in $\GH$ we are going to prove that there exists an element $g\in H^{-1/2}(\partial\Omega)$ such that $\GT(g)=\phi$. To do that, we consider  a continuous linear lifting operator $\GR$ from $H^{1/2}(\partial\Omega)$ into $H^1_0(\Omega)\cap H^2(\Omega)$ satisfying for any $h\in H^{1/2}(\partial\Omega)$
$$
\begin{aligned}
&\GR(h)\in H^1_0(\Omega)\cap H^2(\Omega),\\
& {\cal A}^T\nabla\GR(h)_{|\partial \Omega}\cdot\nu=h\qquad \hbox{on}\quad \partial\Omega,\\
 &||\GR(h)||_{H^2(\Omega)}\le C||h||_{H^{1/2}(\partial\Omega)}.
 \end{aligned}$$ The map $h\longmapsto \ds  \int_\Omega \phi\, \hbox{div}({\cal A}^T\nabla \GR(h))$ is a continuous linear form defined over $H^{1/2}(\partial \Omega)$. Thus, there exists $g\in H^{-1/2}(\partial\Omega)$ such that
\begin{equation}\label{36}
 \int_\Omega \phi\, \hbox{div}({\cal A}^T\nabla \GR(h))=<g,  h >_{H^{-1/2}(\partial\Omega) , H^{1/2}(\partial\Omega)}.
 \end{equation} Since $\phi\in \GH$, we deduce that for any $\psi\in {\cal C}^\infty_0(\Omega)$ we have  $\ds \int_\Omega \phi\, \hbox{div}({\cal A}^T\nabla \psi)=0$. Therefore, for any $\psi\in H^2_0(\Omega)$ we have $\ds \int_\Omega \phi\, \hbox{div}({\cal A}^T\nabla \psi)=0$. Taking into account \eqref{36} we get
$$\int_\Omega \phi\, \hbox{div}({\cal A}^T\nabla\psi)=<g, ( {\cal A}^T\nabla\psi)\cdot\nu>_{H^{-1/2}(\partial\Omega) , H^{1/2}(\partial\Omega)},\qquad  \forall \psi\in H^1_0(\Omega)\cap H^2(\Omega).$$ It yields $\phi=\phi_g$ and then \eqref{30}.
 \end{proof}
  \begin{remark} It is well  known (see e.g. \cite{GDecomp}) that every function $\phi\in\GH$ also belongs to $H^1_\rho(\Omega)$ and verifies
\begin{equation}\label{38}
||\phi||_{\rho}\le C||\phi||_{L^2(\Omega)}.
\end{equation} 
\end{remark}

\section{Error estimates with a non-homogeneous Dirichlet condition}\label{MS}

\begin{theorem}\label{TH1} Let $\big(\phi^\eps\big)_{\eps>0}$ be a sequence of functions belonging to $H^1(\Omega)$ such that 
\begin{equation}\label{CD1}
\begin{aligned}
 &\hbox{div}\big( A_\eps\nabla\phi^\eps\big)=0\qquad \hbox{in}\quad\Omega.
\end{aligned} \end{equation} Setting $g_\eps=\phi^\eps_{|\partial\Omega}$ and  $\phi_{g_\eps}=\GT(g_\eps)\in H^1(\Omega)$, there exists $\eps_0>0$ such that for every $\eps\le \eps_0$  we have
\begin{equation}\label{Estphi}
\begin{aligned} 
& ||\phi^\eps||_{H^1(\Omega)}\le C||g_\eps||_{H^{1/2}(\partial\Omega)},\hskip 15mm ||\phi^\eps-\phi_{g_\eps}||_{L^2(\Omega)}\le C \eps^{1/2}||g_\eps||_{H^{1/2}(\partial\Omega)},\\
&\Big\|\rho\Big(\nabla\phi^\eps-\nabla\phi_{g_\eps}-\sum_{i=1}^n {\cal Q}_\eps\Bigl({\partial\phi_{g_\eps}\over\partial x_i}\Bigr)\nabla_y\chi_i\Big({.\over\eps}\Big)\Big)\Big\|_{L^2(\Omega;\R^n)} \le C \eps^{1/2}||g_\eps||_{H^{1/2}(\partial\Omega)}.
\end{aligned}\end{equation} Moreover we have
\begin{equation}\label{Estphi2}
 ||\phi^\eps||_\rho\le C \big(\eps^{1/2}||g_\eps||_{H^{1/2}(\partial\Omega)}+||g_\eps||_{H^{-1/2}(\partial\Omega)}\big).
\end{equation} The $\chi_i$'s are the correctors introduced in Section \ref{Rhom} and $\GT$ is the operator defined in Section \ref{Op}.
\end{theorem} 
\begin{proof}  {\it Step 1.} We prove the first estimate in \eqref{Estphi}. From  Section \ref{Op}  we get
\begin{equation}\label{Fesp}
||\phi_{g_\eps}||_{H^1(\Omega)}\le C||g_\eps||_{H^{1/2}(\partial\Omega)}\qquad ||\phi_{g_\eps}||_\rho\le C||g_\eps||_{H^{-1/2}(\partial\Omega)}.
\end{equation}

 \noindent We write \eqref{CD1} in the following  weak form:
  \begin{equation}\label{SecVar}
  \begin{aligned}
 &\phi^\eps=\check \phi_\eps+\phi_{g_\eps},\quad \check \phi_\eps\in H^1_0(\Omega) \\ 
 &\int_\Omega A_\eps\nabla \check \phi_\eps\nabla v= - \int_\Omega A_\eps\nabla \phi_{g_\eps}\nabla v \qquad \forall v\in H^1_0(\Omega).
 \end{aligned}\end{equation} 
\noindent The solution $\check \phi_\eps$ of the above variational problem satisfies
\begin{equation*}
||\check \phi_\eps||_{H^1(\Omega)}\le C||\nabla \phi_{g_\eps}||_{L^2(\Omega;\R^n)}.
\end{equation*} Hence, from  $\eqref{Fesp}_1$ and the above estimate we get  the first inequality in \eqref{Estphi}. 
\vskip 1mm
\noindent{\it Step 2. } We prove the second estimate in \eqref{Estphi}.  

\noindent For every test function $v\in H^1_0(\Omega)$ we have
\begin{equation}\label{4000}
\int_\Omega A_\eps\nabla\phi^\eps\nabla v= 0.
\end{equation}
\vskip 1mm
\noindent Now, in order to obtain the $L^2$ error estimate  we proceed as in  the proof of the Theorem 3.2 in \cite{GG2}.  We first recall that for any $\phi\in H^1(\Omega)$ we have (see  Lemma \ref{lemA}) for every $\eps\le \eps_0\doteq \ds{\gamma_0/ 3\sqrt n }$
   \begin{equation*}
 ||\phi||_{L^2(\widetilde{\Omega}_{3c_0\sqrt n\eps})}\le C\eps^{1/2}||\phi||_{H^1(\Omega)}.
 \end{equation*}
\noindent  Let $U$ be a test function belonging to $H^1_0(\Omega)\cap H^2(\Omega)$. The above estimate yields
\begin{equation}\label{phibord2}
||\nabla U||_{L^2(\widetilde{\Omega}_{3c_0\sqrt n\eps};\R^n)}\le C\eps^{1/2}||U||_{H^2(\Omega)}
 \end{equation} which in turn with  \eqref{Def1}-\eqref{T2}-$\eqref{Def2}_1$  and $\eqref{Estphi}_1$-\eqref{4000}  lead to
  \begin{equation}\label{4Int}
 \Big|\int_{\Omega\times Y} A(y){\cal T}_\eps(\nabla\phi^\eps)(x,y)\nabla U(x)dxdy\Big|\le C\eps^{1/2}||g_\eps||_{H^{1/2}(\partial\Omega)}||U||_{H^2(\Omega)}.
 \end{equation} The Theorem 2.3 in \cite{GG2} gives an element $\widehat{\phi}_\eps\in L^2(\Omega; H^1_{per}(Y))$ such that
\begin{equation}\label{Proj}
\begin{aligned}
||{\cal T}(\nabla\phi^\eps)-\nabla\phi^\eps-\nabla_y\widehat{\phi}_\eps||_{[L^2(Y; (H^1(\Omega))^{'})]^n}
&\le C\eps^{1/2}||\nabla \phi^\eps||_{L^2(\Omega;\R^n)}\\
&\le C\eps^{1/2}||g_\eps||_{H^{1/2}(\partial\Omega)}.
\end{aligned}\end{equation} The above inequalities   \eqref{4Int} and \eqref{Proj} yield
 \begin{equation}\label{Etap1}
 \Big|\int_{\Omega\times Y} A \big(\nabla\phi^\eps+\nabla_y\widehat{\phi}_\eps\big)\nabla U \Big|\le C\eps^{1/2}||g_\eps||_{H^{1/2}(\partial\Omega)}||U||_{H^2(\Omega)}.
 \end{equation} 
We set 
\begin{equation*}\label{reps}
\forall x\in \R^n,\qquad \rho_\eps(x)=\inf\Big\{1,{\rho(x)\over \eps}\Big\}.
\end{equation*} Now, we take $\overline{\chi}\in H^1_{per}(Y)$ and we consider the test function $u_\eps\in H^1_0(\Omega)$  defined for a.e. $x\in \Omega$ by 
 $$u_\eps(x)=\eps\rho_\eps(x){\cal Q}_\eps\Big({\partial U\over \partial x_i}\Big)(x)\overline{\chi}\Big({x\over \eps}\Big).$$  Due to  $\eqref{820}_2$ and  \eqref{phibord2} we get
 \begin{equation}\label{EE1}
 \Big\|{\cal Q}_\eps\Big({\partial U\over \partial x_i}\Big)\nabla_y\overline{\chi}\Big({\cdot\over \eps}\Big)\Big\|_{L^2(\widetilde{\Omega}_{\sqrt n\eps};\R^n)} \le C\eps^{1/2}||U||_{H^2(\Omega)}||\overline{\chi}||_{H^1(Y)}
\end{equation}
Then by a straightforward calculation and thanks to  $\eqref{820}_2$-$\eqref{8200}_2$  and \eqref{phibord2}-\eqref{EE1} we obtain
\begin{equation*}
\Big\|\nabla u_\eps-{\cal Q}_\eps\Big({\partial U\over \partial x_i}\Big)\nabla_y\overline{\chi}\Big({\cdot\over \eps}\Big)\Big\|_{L^2(\Omega;\R^n)} \le C\eps^{1/2}||U||_{H^2(\Omega)}||\overline{\chi}||_{H^1(Y)}
\end{equation*}  which in turn with again \eqref{EE1} give 
\begin{equation}\label{EE2}
 \|\nabla u_\eps\|_{L^2(\widetilde{\Omega}_{\sqrt n\eps};\R^n)} \le C\eps^{1/2}||U||_{H^2(\Omega)}||\overline{\chi}||_{H^1(Y)}
 \end{equation} and then with  $\eqref{8200}_1$  they lead to
$$\Big\|\nabla u_\eps-{\cal M}_\eps\Big({\partial U\over \partial x_i}\Big)\nabla_y\overline{\chi}\Big({\cdot\over \eps}\Big)\Big\|_{L^2(\Omega;\R^n)}  \le C\eps^{1/2}||U||_{H^2(\Omega)}||\overline{\chi}||_{H^1(Y)}.$$ In \eqref{4000} we replace $\nabla u_\eps$ with $\ds {\cal M}_\eps\Big({\partial U\over \partial x_i}\Big)\nabla_y\overline{\chi}\Big({\cdot\over \eps}\Big)$; we continue using   \eqref{Def1}-\eqref{T2} and $\eqref{Estphi}_1$-\eqref{EE2} to obtain
 \begin{equation*}
\Big|\int_{\Omega\times Y} A(y){\cal T}_\eps(\nabla\phi^\eps)(x,y){\cal M}_\eps\Big({\partial U\over \partial x_i}\Big)(x)\nabla_y\overline{\chi}(y)dxdy\Big|\le C\eps^{1/2}||g_\eps||_{H^{1/2}(\partial\Omega)}||U||_{H^2(\Omega)}||\overline{\chi}||_{H^1(Y)}
 \end{equation*} which with $\eqref{M1}_2$ and then  \eqref{Proj} give
  \begin{equation}\label{00}
\Big|\int_{\Omega\times Y} A(y)\big(\nabla\phi^\eps(x)+\nabla_y\widehat{\phi}_\eps(x,y)\big){\partial U\over \partial x_i}(x)\nabla_y\overline{\chi}(y)dx \, dy \Big|\le C\eps^{1/2}||g_\eps||_{H^{1/2}(\partial\Omega)}||U||_{H^2(\Omega)}||\overline{\chi}||_{H^1(Y)}.
 \end{equation} As in \cite{GG2} we introduce the adjoint correctors $\overline{\chi}_i\in H^1_{per}(Y)$, $i\in\{1,\ldots,n\}$, defined by
\begin{equation}\label{corbar}
\int_Y A(y) \nabla_y \psi(y)\nabla_y(\overline{\chi}_i(y)+y_i )dy=0\qquad \forall \psi\in H^1_{per}(Y).
\end{equation} From \eqref{00} we get
  \begin{equation*}
\Big|\int_{\Omega\times Y} A\big(\nabla\phi^\eps+\nabla_y\widehat{\phi}_\eps\big)\nabla_y\Big(\sum_{i=1}^n{\partial U\over \partial x_i}\overline{\chi}_i\Big) \Big|\le C\eps^{1/2}||g_\eps||_{H^{1/2}(\partial\Omega)}||U||_{H^2(\Omega)}
 \end{equation*}  and from the definition \eqref{cor} of the correctors $\chi_i$ we have
 \begin{equation*}
\int_{\Omega\times Y} A\Big(\nabla\phi^\eps+\sum_{i=1}^n{\partial\phi^\eps\over \partial x_i}\nabla_y\chi_i\Big)\nabla_y\Big(\sum_{j=1}^n{\partial U\over \partial x_j}\overline{\chi}_j\Big)=0.
 \end{equation*} Thus
  \begin{equation*}
\Big|\int_{\Omega\times Y} A\nabla_y\Big( \widehat{\phi}_\eps-\sum_{i=1}^n{\partial\phi^\eps\over \partial x_i}\chi_i\Big)\nabla_y\Big(\sum_{j=1}^n{\partial U\over \partial x_j}\overline{\chi}_j\Big) \Big|\le C\eps^{1/2}||g_\eps||_{H^{1/2}(\partial\Omega)}||U||_{H^2(\Omega)}
 \end{equation*} and thanks to \eqref{corbar} we obtain
   \begin{equation*}
\Big|\int_{\Omega\times Y} A\nabla_y\Big( \widehat{\phi}_\eps-\sum_{i=1}^n{\partial\phi^\eps\over \partial x_i}\chi_i\Big)\nabla U \Big|\le C\eps^{1/2}||g_\eps||_{H^{1/2}(\partial\Omega)}||U||_{H^2(\Omega)}.
 \end{equation*} The above estimate, \eqref{Etap1} and the expression \eqref{Aij} of the matrix ${\cal A}$  yield
$$\Big|\int_\Omega{\cal A}\nabla\phi^\eps\nabla U \Big|\le C\eps^{1/2}||g_\eps||_{H^{1/2}(\partial\Omega)}||U||_{H^2(\Omega)}.$$ Finally, since we have $\ds \int_\Omega{\cal A}\nabla\phi_{g_\eps}\nabla v=0$ for any $v\in H^1_0(\Omega)$, we deduce that
$$\forall U\in H^1_0(\Omega)\cap H^2(\Omega),\qquad \Big|\int_\Omega{\cal A}\nabla (\phi^\eps-\phi_{g_\eps})\nabla U \Big|\le C\eps^{1/2}||g_\eps||_{H^{1/2}(\partial\Omega)}||U||_{H^2(\Omega)}.$$ 
Now, let $U_\eps\in H^1_0(\Omega)$ be the solution of the following variational problem:
$$\int_\Omega{\cal A}\nabla v\nabla U_\eps=\int_\Omega v(\phi^\eps-\phi_{g_\eps}),\qquad \forall v\in H^1_0(\Omega).$$ Under the assumption on the boundary of $\Omega$, we know  that $U_\eps$ belongs to $H^1_0(\Omega)\cap H^2(\Omega)$ and satisfies $||U_\eps||_{H^2(\Omega)}\le C||\phi^\eps-\phi_{g_\eps}||_{L^2(\Omega)}$ (the constant do not depend on $\eps$). Therefore,  the second estimate in \eqref{Estphi} is proved.
\vskip 1mm
\noindent{\it Step 3. } We prove the third estimate in \eqref{Estphi} and \eqref{Estphi2}.  The partial derivative $\ds{\partial \phi_{g_\eps}\over \partial x_i}$ satisfies
 $$\hbox{div}\Big({\cal A}\nabla\big({\partial \phi_{g_\eps}\over \partial x_i}\big)\Big)=0\qquad \hbox{in}\quad \Omega,\qquad {\partial \phi_{g_\eps}\over \partial x_i}\in L^2(\Omega).$$ Thus, from Remark \ref{38} and estimate  $\eqref{Fesp}_2$ we get
\begin{equation}\label{Est2}
\Big\|\rho \nabla\Big({\partial \phi_{g_\eps}\over \partial x_i}\Big)\Big\|_{L^2(\Omega;\R^n)}\le C\Big\|{\partial \phi_{g_\eps}\over \partial x_i}\Big\|_{L^2(\Omega)}\le C||g_\eps||_{H^{1/2}(\partial\Omega)}.
\end{equation}
\noindent Now, let $U$ be in $H^1_0(\Omega)$, the function $\rho U$ belongs to $H^1_{1/\rho}(\Omega)$. Applying the Theorem \ref{Theorem 2.3 : } with the function $\rho U$, there exists $\widehat{u}_\eps\in L^2(\Omega ; H^1_{per}(Y))$ such that
\begin{equation}\label{EE3}
||{\cal T}_\eps(\nabla(\rho U))-\nabla (\rho U)-\nabla_y\widehat{u}_\eps||_{L^2(Y; (H^1_\rho(\Omega;\R^n))^{'})} \le C\eps||\rho U||_{H^1_{1/\rho}(\Omega)}\le C\eps||U||_{H^1(\Omega)}.
\end{equation}
The above estimates \eqref{Est2} and \eqref{EE3} lead to 
$$\Big|\int_{\Omega\times Y}A\Big(\nabla\phi_{g_\eps}+\sum_{i=1}^n{\partial\phi_{g_\eps}\over \partial x_i}\nabla_y\chi_i\Big)\Big({\cal T}_\eps\big(\nabla( \rho U)\big)-\nabla (\rho U)-\nabla_y\widehat{u}_\eps\Big)\Big|\le C\eps||U||_{H^1(\Omega)}||g_\eps||_{H^{1/2}(\partial\Omega)}$$ By definition of the correctors $\chi_i$ we have
$$\int_{\Omega\times Y}A\Big(\nabla\phi_{g_\eps}+\sum_{i=1}^n{\partial\phi_{g_\eps}\over \partial x_i}\nabla_y\chi_i\Big)\nabla_y\widehat{u}_\eps=0.$$ Besides, from the definitions of the function  $\phi_{g_\eps}$  and the homogenized matrix ${\cal A}$ we have
$$0=\int_{\Omega}{\cal A} \nabla\phi_{g_\eps}\nabla(\rho U)=\int_{\Omega\times Y}A \Big(\nabla\phi_{g_\eps}+\sum_{i=1}^n{\partial\phi_{g_\eps}\over \partial x_i}\nabla_y\chi_i\Big)\nabla(\rho U).$$ The above inequality and equalities yield
\begin{equation}\label{EE4}
\Big|\int_{\Omega\times Y}A\Big(\nabla\phi_{g_\eps}+\sum_{i=1}^n{\partial\phi_{g_\eps}\over \partial x_i}\nabla_y\chi_i\Big){\cal T}_\eps\big(\nabla(\rho U)\big)\Big|\le C\eps||\nabla U||_{L^2(\Omega;\R^n)}||g_\eps||_{H^{1/2}(\partial\Omega)}.
\end{equation} We have
$$\nabla(\rho U)=\rho\Big(\nabla U+\nabla \rho{U\over \rho}\Big).$$ 
Then since $U/\rho\in L^2(\Omega)$ and $||U/\rho||_{L^2(\Omega)}\le C||\nabla U||_{L^2(\Omega;\R^n)}$ and due to \eqref{Trho} we get
$$\Big\|{\cal T}_\eps\big(\nabla(\rho U)\big)-\rho{\cal T}_\eps\Big(\nabla U+\nabla \rho{U\over \rho}\Big)\Big\|_{L^2(\Omega;\R^n)}\le C\eps\Big\|\nabla U+\nabla \rho{U\over \rho}\Big\|_{L^2(\Omega;\R^n)}\le C\eps||U||_{H^1(\Omega)}.$$ From \eqref{EE4} and the above inequalities  we deduce that
$$
\begin{aligned}
\Big|\int_{\Omega\times Y}A\Big(\rho\nabla\phi_{g_\eps}+\sum_{i=1}^n\rho{\partial\phi_{g_\eps}\over \partial x_i}\nabla_y\chi_i\Big){\cal T}_\eps\Big(\nabla U+\nabla \rho{U\over \rho}\Big)\Big|\le C\eps||\nabla U||_{L^2(\Omega;\R^n)}||g_\eps||_{H^{1/2}(\partial\Omega)}.
\end{aligned}$$ We recall that $\rho \nabla\phi_{g_\eps}\in H^1_0(\Omega;\R^n)$, hence from  $\eqref{Def2}_2$,  $\eqref{M1}_1$ and  \eqref{Est2} we get
\begin{equation*}
\begin{aligned}
\Big|&\int_{\Omega\times Y}A\Big(\rho\nabla\phi_{g_\eps}+\sum_{i=1}^n\rho{\partial\phi_{g_\eps}\over \partial x_i}\nabla_y\chi_i\Big){\cal T}_\eps\Big(\nabla U+\nabla \rho{U\over \rho}\Big)\\
-&\int_{\Omega\times Y}A  \Big({\cal T}_\eps(\rho\nabla\phi_{g_\eps})+\sum_{i=1}^n{\cal M}_\eps\Big(\rho{\partial\phi_{g_\eps}\over \partial x_i}\Big)\nabla_y\chi_i\Big){\cal T}_\eps\Big(\nabla U+\nabla \rho{U\over \rho}\Big)\Big|\le C\eps||\nabla U||_{L^2(\Omega;\R^n)}||g_\eps||_{H^{1/2}(\partial\Omega)}.
\end{aligned}\end{equation*} 
 Then  transforming by inverse unfolding  we obtain
 \begin{equation*}
\begin{aligned}
\Big|\int_{\widehat{\Omega}_\eps}A_\eps  \Big(\rho \nabla\phi_{g_\eps}+\sum_{i=1}^n{\cal M}_\eps\Big(\rho{\partial\phi_{g_\eps}\over \partial x_i}\Big)\nabla_y\chi_i\big({\cdot\over \eps}\big)\Big)\Big(\nabla U+\nabla \rho{U\over \rho}\Big)\Big|\le C\eps||\nabla U||_{L^2(\Omega;\R^n)}||g_\eps||_{H^{1/2}(\partial\Omega)}.
\end{aligned}\end{equation*} 
Now, thanks to  \eqref{820RR}  and \eqref{Est2} we get
\begin{equation*}
\Big|\int_{\Omega}A_\eps  \rho\Big(\nabla\phi_{g_\eps}+\sum_{i=1}^n{\cal M}_\eps\Big({\partial\phi_{g_\eps}\over \partial x_i}\Big)\nabla_y\chi_i\big({\cdot\over \eps}\big)\Big)\Big(\nabla U+\nabla \rho{U\over \rho}\Big)\Big|\le C\eps||\nabla U||_{L^2(\Omega;\R^n)}||g_\eps||_{H^{1/2}(\partial\Omega)}.
\end{equation*} Then using $\eqref{820RR1}_1$ it leads to
\begin{equation*}
\Big|\int_{\Omega}A_\eps  \Big(\nabla\phi_{g_\eps}+\sum_{i=1}^n{\cal Q}_\eps\Big({\partial\phi_{g_\eps}\over \partial x_i}\Big)\nabla_y\chi_i\big({\cdot\over \eps}\big)\Big)\nabla (\rho U)\Big|\le C\eps||\nabla U||_{L^2(\Omega;\R^n)}||g_\eps||_{H^{1/2}(\partial\Omega)}.
\end{equation*}
We recall that $\ds \int_{\Omega}A_\eps  \nabla\phi^\eps\nabla(\rho U)=0$. We choose $\ds U=\rho\Big(\phi^\eps-\phi_{g_\eps}-\eps \sum_{i=1}^n{\cal Q}_\eps\Big({\partial\phi_{g_\eps}\over \partial x_i}\Big)\chi_i\big({\cdot\over \eps}\big)\Big)$ which belongs to $H^1_0(\Omega)$. Due to  the second estimate in \eqref{Estphi}, the third one  in \eqref{Estphi} follows immediately. 

\noindent The estimate \eqref{Estphi2} is the consequence of $\eqref{820RR1}_2$, $\eqref{Estphi}_2$, $\eqref{Estphi}_3$, $\eqref{Fesp}_2$ and \eqref{Est2}.
\end{proof} 

\begin{cor}  Let $\big(\phi^\eps\big)_{\eps>0}$ be a sequence of functions belonging to $H^1(\Omega)$ and satisfying \eqref{CD1}. We set
 $g_\eps=\phi^\eps_{|\partial\Omega}$, if we have
\begin{equation*}\label{H1}
g_\eps\rightharpoonup g\quad \hbox{weakly in } \quad H^{1/2}(\partial \Omega)
\end{equation*} then we obtain
\begin{equation}\label{C1}
\begin{aligned}
&\phi^\eps\rightharpoonup \phi_g\quad \hbox{weakly in } \quad H^1(\Omega),\\
&\phi^\eps-\phi_g-\eps\sum_{i=1}^n {\cal Q}_\eps\Bigl({\partial\phi_g\over\partial x_i}\Bigr)\chi_i\Big({.\over\eps}\Big)\longrightarrow 0\quad \hbox{strongly in } \quad H^1_\rho(\Omega).
\end{aligned}\end{equation} Moreover, if
\begin{equation}\label{H1}
g_\eps\longrightarrow g\quad \hbox{strongly in } \quad H^{1/2}(\partial \Omega)
\end{equation} then we have
\begin{equation}\label{C10}
\phi^\eps-\phi_g-\eps\sum_{i=1}^n {\cal Q}_\eps\Bigl({\partial\phi_g\over\partial x_i}\Bigr)\chi_i\Big({.\over\eps}\Big)\longrightarrow 0\quad \hbox{strongly in } \quad H^1(\Omega).
\end{equation} 
\end{cor} 
\begin{proof} Thanks to $\eqref{Estphi}_1$ the sequence $\big(\phi^\eps\big)_{\eps>0}$ is uniformly bounded in  $H^1(\Omega)$. 
Then  due to Lemma \ref{GT} and Remark \ref{38} we get
$$||\phi_g-\phi_{g_\eps}||_\rho\le C||g-g_\eps||_{H^{-1/2}(\partial\Omega)}$$ which with $\eqref{Estphi}_2$ (resp. $\eqref{Estphi}_3$) give the convergence $\eqref{C1}_1$ (resp. $\eqref{C1}_2$).

\noindent Under the assumption \eqref{H1}, we use \eqref{phifinit} and we proceed as in the proof of the Theorem 6.1 of \cite {CDG} in order  to obtain the strong convergence \eqref{C10}.
\end{proof}

\begin{theorem}\label{TH3} Let $\phi^\eps$ be the solution of the following homogenization problem:
\begin{equation*}
-\hbox{div}\big( A_\eps\nabla\phi^\eps)=f\quad \hbox{in}\quad\Omega,\qquad 
\phi^\eps=g \quad \hbox{on }\quad \partial\Omega
\end{equation*} where $f\in L^2(\Omega)$ and $g\in H^{1/2}(\partial \Omega)$. We have
\begin{equation*}
\begin{aligned} 
&  ||\phi^\eps-\Phi||_{L^2(\Omega)} \le C \big\{\eps||f||_{L^2(\Omega)}+\eps^{1/2}||g||_{H^{1/2}(\partial\Omega)}\big\},\\
&\Big\|\rho\Big(\nabla\phi^\eps-\nabla\Phi-\sum_{i=1}^n {\cal Q}_\eps\Bigl({\partial\Phi\over\partial x_i}\Bigr)\nabla_y\chi_i\Big({.\over\eps}\Big)\Big)\Big\|_{L^2(\Omega;\R^n)} \le C \big\{\eps||f||_{L^2(\Omega)}+\eps^{1/2}||g||_{H^{1/2}(\partial\Omega)}\big\}
\end{aligned}\end{equation*} where $\Phi$ is the solution of the homogenized problem
\begin{equation*}
-\hbox{div}\big( {\cal A}\nabla\Phi \big)=f\quad \hbox{in}\quad\Omega,\qquad
\Phi=g \quad \hbox{on }\quad \partial\Omega.
\end{equation*} Moreover we have
\begin{equation}\label{EstTH3}
\phi^\eps-\Phi-\eps\sum_{i=1}^n {\cal Q}_\eps\Bigl({\partial\Phi\over\partial x_i}\Bigr)\chi_i\Big({.\over\eps}\Big)\longrightarrow 0\quad \hbox{strongly in }\quad H^1(\Omega). 
\end{equation}
\end{theorem} 
 \begin{proof} Let $\widetilde{\phi}^\eps$ be the solution of the homogenization problem
\begin{equation*}
 \widetilde{\phi}^\eps\in H^1_0(\Omega),\qquad -\hbox{div}\big( A_\eps\nabla\widetilde{\phi}^\eps \big)=f\qquad \hbox{in}\quad\Omega
 \end{equation*} and $\widetilde{\Phi}$  the solution of the homogenized problem
\begin{equation*}
 \widetilde{\Phi}\in H^1_0(\Omega),\qquad -\hbox{div}\big( {\cal A}\nabla\widetilde{\Phi} \big)=f\qquad \hbox{in}\quad\Omega.
 \end{equation*} The Theorem 3.2 in \cite{GG2} gives the following estimate:
\begin{equation}\label{Esttildephi}
 ||\widetilde{\phi}^\eps-\widetilde{\Phi}||_{L^2(\Omega)}+ \Bigl\|\rho\nabla\Bigl(\widetilde{\phi}^\eps-\widetilde{\Phi}-\eps \sum_{i=1}^n  {\cal Q}_\eps\Bigl({\partial\widetilde{\Phi}\over\partial x_i}\Bigr)\chi_i\Big({.\over\eps}\Big)\Bigr)
\Bigr\|_{L^2(\Omega;\R^n)}\le C\eps||f||_{L^2(\Omega)}
\end{equation}  while the Theorem 4.1 in \cite{GG1} gives
\begin{equation}\label{Esttildephi2}
\Bigl\|\widetilde{\phi}^\eps-\widetilde{\Phi}-\eps \sum_{i=1}^n  {\cal Q}_\eps\Bigl({\partial\widetilde{\Phi}\over\partial x_i}\Bigr)\chi_i\Big({.\over\eps}\Big)\Bigr\|_{H^1(\Omega)}\le C\eps^{1/2}||f||_{L^2(\Omega)}.
\end{equation} The function $\phi^\eps-\widetilde{\phi}^\eps$ satisfies
$$\hbox{div}\big( A_\eps\nabla(\phi^\eps-\widetilde{\phi}^\eps)\big)=0\qquad \hbox{in}\quad\Omega,\qquad 
\phi^\eps-\widetilde{\phi}^\eps=g \qquad \hbox{on }\quad \partial\Omega.$$ Thanks to the inequalities \eqref{Estphi} and \eqref{Esttildephi} we deduce the estimates of the theorem. The strong convergence \eqref{EstTH3} is a consequence of \eqref{Esttildephi2} and the strong convergence \eqref{C10} after having observed that $\Phi-\widetilde{\Phi}=\phi_g$.
\end{proof}
\begin{remark} In Theorem \ref{TH3}, if $g\in H^{3/2}(\partial \Omega)$ then in the estimates therein, we can replace $\eps^{1/2}||g||_{H^{1/2}(\partial\Omega)}$ with $\eps||g||_{H^{3/2}(\partial\Omega)}$. Moreover we have the following $H^1$-global error estimate:
$$\Big\|\phi^\eps- \Phi-\eps\sum_{i=1}^n {\cal Q}_\eps\Bigl({\partial\Phi\over\partial x_i}\Bigr)\chi_i\Big({.\over\eps}\Big)\Big\|_{H^1(\Omega)} \le C \eps^{1/2}\big\{||f||_{L^2(\Omega)}+||g||_{H^{3/2}(\partial\Omega)}\big\}.$$
\end{remark}

\section{A first result with strongly oscillating boundary data}

\noindent In this section we consider the solution $\phi^\eps$ of the homogenization  problem 
\begin{equation}\label{SecPb}
 \begin{aligned}
& \hbox{div}\big( A_\eps\nabla\phi^\eps)=0\qquad \hbox{in}\quad\Omega\\
&\phi^\eps=g_\eps\qquad \hbox{on}\quad \partial\Omega \\ 
\end{aligned}\end{equation} 
where  $g_\eps\in H^{1/2}(\partial\Omega)$. As a consequence of the Theorem \ref{TH1} we obtain the following result:
\begin{theorem}\label{TH65} Let $\phi^\eps$ be the solution of the problem \eqref{SecPb}. If we have
\begin{equation*}
g_\eps\rightharpoonup g\quad \hbox{weakly in } \quad H^{-1/2}(\partial \Omega)
\end{equation*} and
\begin{equation}\label{H2}
\eps^{1/2}g_\eps \longrightarrow 0\quad \hbox{strongly in } \quad H^{1/2}(\partial \Omega)
\end{equation} then
\begin{equation}\label{C22}
\phi^\eps\rightharpoonup \phi_g \quad \hbox{weakly in } \quad H^1_\rho(\Omega).
\end{equation} Furthermore, if we have
\begin{equation*}
g_\eps\longrightarrow g\quad \hbox{strongly in } \quad H^{-1/2}(\partial \Omega) 
\end{equation*} then
\begin{equation}\label{C2}
\phi^\eps-\phi_g - \eps\sum_{i=1}^n {\cal Q}_\eps\Bigl({\partial\phi_{g_\eps}\over\partial x_i}\Bigr)\chi_i\Big({\cdot\over \eps}\Big)\longrightarrow 0 \quad \hbox{strongly in } \quad H^1_\rho(\Omega).
\end{equation} 

\end{theorem}
\begin{proof} Due to \eqref{Estphi2} the sequence $\big(\phi^\eps)_{\eps>0}$ is uniformly bounded in $H^1_\rho(\Omega)$.
From the estimates $\eqref{Estphi}_3$ and $\eqref{Fesp}_2$ we get
$$\Big\|\phi^\eps-\phi_{g_\eps} - \eps\sum_{i=1}^n {\cal Q}_\eps\Bigl({\partial\phi_{g_\eps}\over\partial x_i}\Bigr)\chi_i\Big({\cdot\over \eps}\Big)\Big\|_{H^1_\rho(\Omega)}\le C\eps^{1/2}||g_\eps||_{H^{1/2}(\partial\Omega)}.$$  Then  using the variational problem \eqref{PBphig}  and estimate $\eqref{Fesp}_2$  we obtain 
$$\phi_{g_\eps}\rightharpoonup \phi_g \quad \hbox{weakly in } \quad H^1_\rho(\Omega).$$
Since the sequence  $\ds  \eps\sum_{i=1}^n {\cal Q}_\eps\Bigl({\partial\phi_{g_\eps}\over\partial x_i}\Bigr)\chi_i\big({\cdot\over \eps}\big)$  is uniformly bounded in $H^1_\rho(\Omega)$ and strongly converges to $0$ in $L^2(\Omega)$,  we have $\ds  \eps\sum_{i=1}^n {\cal Q}_\eps\Bigl({\partial\phi_{g_\eps}\over\partial x_i}\Bigr)\chi_i\big({\cdot\over \eps}\big)\rightharpoonup 0$ weakly in $H^1_\rho(\Omega)$. Therefore the weak convergence \eqref{C22} is proved.
\vskip 1mm
 In the case $g_\eps\longrightarrow g$ strongly in $H^{-1/2}(\partial \Omega)$, the estimates \eqref{T} and  \eqref{38} lead to
$$||\phi_{g_\eps}-\phi_g||_{ H^1_{\rho}(\Omega)}\le C||g_\eps-g||_{H^{-1/2}(\partial\Omega)}.$$ Hence with $\eqref{820RR1}_2$ they yield  \eqref{C2}.
\end{proof}
\noindent In a forthcoming paper we will show that  in both cases (weak or strong convergence of the sequence $(g_\eps)_{\eps>0}$ towards  $g$  in $H^{-1/2}(\partial \Omega)$) the  assumption \eqref{H2} is essential in order to obtain at least \eqref{C22}.

\end{document}